\documentclass[11pt]{article}
\usepackage{amsfonts,amsmath,amssymb,mathrsfs,bm,amsthm}
\usepackage{graphicx}
\usepackage{tikz,wrapfig}
\usepackage{caption,epic,eepic}
\usepackage{hyperref}
\usepackage{xcolor}
\DeclareFontFamily{U}{mathc}{}
\DeclareFontShape{U}{mathc}{m}{it}{<->s*[1.03] mathc10}{}
\DeclareMathAlphabet{\mathcal}{U}{mathc}{m}{it}
\usepackage[scr=rsfso,calscaled=1.2]{mathalfa}

\newcommand{\noi}{\noindent}

\newcommand{\red}[1]{\textcolor{red}{#1}}

\newtheorem{theorem}{Theorem}[section]
\newtheorem{proposition}[theorem]{Proposition}
\newtheorem{corollary}[theorem]{Corollary}
\newtheorem{conjecture}[theorem]{Conjecture}
\newtheorem{lemma}[theorem]{Lemma}

\theoremstyle{definition}
\newtheorem{remark}[theorem]{Remark}
\newtheorem{definition}[theorem]{Definition}

\newcommand{\bd}{\begin{definition}}
\newcommand{\ed}{\end{definition}}
\newcommand{\bt}{\begin{theorem}}
\newcommand{\et}{\end{theorem}}
\newcommand{\bl}{\begin{lemma}}
\newcommand{\el}{\end{lemma}}
\newcommand{\bp}{\begin{proposition}}
\newcommand{\ep}{\end{proposition}}
\newcommand{\bcor}{\begin{corollary}}
\newcommand{\ecor}{\end{corollary}}
\newcommand{\br}{\begin{remark}\rm}
\newcommand{\er}{\end{remark}}
\newcommand{\bcon}{\begin{conjecture}}
\newcommand{\econ}{\end{conjecture}}

\renewcommand{\theequation}{\thesection .\arabic{equation}}

\newcommand{\Ei}{{\cal E}}

\newcommand{\Oi}{{\cal O}}

\newcommand{\R}{{\mathbb R}}
\newcommand{\N}{{\mathbb N}}
\newcommand{\Z}{{\mathbb Z}}

\renewcommand{\P}{{\mathbb P}}
\newcommand{\E}{{\mathbb E}}

\newcommand\bP{\ensuremath{\boldsymbol{\mathrm{P}}}}
\newcommand\bE{\ensuremath{\boldsymbol{\mathrm{E}}}}

\newcommand\be{\ensuremath{\boldsymbol{\mathrm{e}}}}

\newcommand{\dd}{{\rm d}}

\renewcommand{\tilde}{\widetilde}          

\newcommand{\relmiddle}[1]{\mathrel{}\middle#1\mathrel{}}
\newcommand{\cond}{\relmiddle{|}}
\newcommand{\ball}[2]{B\left(#1;#2\right)}
\newcommand{\Lball}[2]{\mathcal{B}\left(#1;#2\right)}
\newcommand{\fregion}{{B(0;2N) \setminus \Oi}}

\newcommand{\rad}{\varrho_N}
\newcommand{\cent}{\mathcal{x}_N}

\let\originalleft\left
\let\originalright\right
\renewcommand{\left}{\mathopen{}\mathclose\bgroup\originalleft}
\renewcommand{\right}{\aftergroup\egroup\originalright}

\begin{document}

\makeatletter\@addtoreset{equation}{section}
\makeatother\def\theequation{\thesection.\arabic{equation}}

\title{Biased random walk conditioned on survival among Bernoulli obstacles: subcritical phase}
\author{Jian Ding$^{\,1}$ \and Ryoki Fukushima$^{\,2}$ \and Rongfeng Sun$^{\,3}$ \and Changji Xu$^{\,4}$}

\date{\today}

\maketitle

\footnotetext[1]{Statistics Department, University of Pennsylvania. Email: dingjian@wharton.upenn.edu}

\footnotetext[2]{Research Institute for Mathematical Sciences, Kyoto University. Email: ryoki@kurims.kyoto-u.ac.jp}

\footnotetext[3]{Department of Mathematics, National University of Singapore. Email: matsr@nus.edu.sg}

\footnotetext[4]{Department of Statistics, University of Chicago. Email: changjixu@galton.uchicago.edu}

\begin{abstract}
We consider a discrete time biased random walk conditioned to avoid Bernoulli obstacles on $\Z^d$ ($d\geq 2$) up to time $N$. This model is known to undergo a phase transition: for a large bias, the walk is ballistic whereas for a small bias, it is sub-ballistic. We prove that in the sub-ballistic phase, the random walk is contained in a ball of radius $O(N^{1/(d+2)})$, which is the same scale as for the unbiased case. As an intermediate step, we also prove large deviation principles for the endpoint distribution for the unbiased random walk at scales between $N^{1/(d+2)}$ and $o(N^{d/(d+2)})$.
These results improve and complement earlier work by Sznitman [{\em Ann. Sci. \'Ecole Norm. Sup. (4)}, 28(3):345--370, 371--390, 1995].
\end{abstract}

\vspace{.3cm}

\noi

\noi

\section{Introduction}
\subsection{Model and main results}
\label{sec:results}
Let $(S:=(S_n)_{n\ge 0}, \bP)$ be a simple symmetric random walk on $\Z^d$ starting at the origin and denote the corresponding expectation by $\bE$.
When we start the random walk from $x\in\Z^d\setminus\{0\}$, we indicate the starting point by subscript as $\bP_x$ or $\bE_x$.
We place an obstacle at each site $x\in\Z^d$ independently with probability $1-p$ for some $p\in(0,1)$ and write $\Oi$ for the set of sites occupied by the obstacles. Probability and expectation for the random obstacles configuration will be denoted by $\P$ and $\E$, respectively. For a random variable $X$ and an event $A$, we write $\E[X:A]$ for $\E[X \cdot 1_A]$, and this convention applies to other probability measures. We are interested in the behavior of the random walk with bias $h\in\R^d$ conditioned to avoid $\Oi$ for a long time, that is, the hitting time $\tau_\Oi$ of $\Oi$ is large.
\begin{definition}
The annealed law with bias $h \in \R^d$ is defined by
\begin{equation}
 \mu_N^h ((S,\Oi)\in\cdot)=\frac{\E\otimes\bE\left[e^{\langle h, S_N\rangle}\colon \tau_{\cal{O}}>N, (S,\Oi)\in\cdot\right]}{\E\otimes\bE\left[e^{\langle h,S_N\rangle}\colon \tau_{\cal{O}}>N\right]}.
\label{eq:biased_law}
\end{equation}
When $h =0$, we omit the superscript and write $\mu_N^0=\mu_N$ for simplicity.
\end{definition}
\begin{remark}
In the definition of $\mu_N^h$, we can perform the $\E$-expectation conditionally on the random walk to get the expression
\begin{equation}
 \mu_N^h (S\in \cdot)=\frac{\bE\left[\exp\left\{\langle h,S_N\rangle-|S_{[0,N]}|\log \tfrac{1}{p}\right\}\colon S\in\cdot\right]}{\bE\left[\exp\left\{\langle h,S_N\rangle-|S_{[0,N]}|\log \tfrac{1}{p}\right\}\right]},
\end{equation}
where $S_{[0,N]}=\{S_0,S_1,\dotsc,S_N\}$ is the range of the random walk.
This can be viewed as a model of self-attractive polymer with an external force $h$.
\end{remark}
In the case $h=0$, the leading order asymptotics of the partition function was determined by Donsker--Varadhan~\cite{DV79} as follows:
\begin{equation}
\begin{split}
 \P\otimes\bP\left(\tau_\Oi>N\right)&=\exp\left\{-c(d,p)N^{\frac{d}{d+2}}+o(N^{\frac{d}{d+2}})\right\}\text{ with }\\
 c(d,p)&=\inf_{\boldsymbol{U}\subset \R^d}\left\{\boldsymbol{\lambda_U}+{\rm vol}(\boldsymbol{U}) \log\frac{1}{p}\right\}
\end{split}
\label{eq:DV79}
\end{equation}
as $N\to\infty$, where $\boldsymbol{\lambda_U}$ denotes the smallest eigenvalue of the continuum Laplacian $-\tfrac{1}{2d}\Delta$ with the Dirichlet boundary condition outside $\boldsymbol{U}\subset\R^d$. Here and in what follows, we use boldface to denote subsets of $\R^d$ and the eigenvalues of continuum Laplacian. 
For instance, we write $\boldsymbol{B}(x;r)\subset\R^d$ for the Euclidean ball with center $x$ and radius $r$ and $B(x;r):=\boldsymbol{B}(x;r)\cap \Z^d$.
By the classical Faber--Krahn inequality, the above infimum is achieved by 
$U=\boldsymbol{B}(0;\varrho_1)$ 
for some $\varrho_1=\varrho_1(d,p)$ (but in fact the center is arbitrary). This indicates that the best strategy to achieve $\{\tau_\Oi>N\}$ is for the random walk to spend most of the time in a vacant (i.e.,~free of obstacles) ball of radius
\begin{equation}
 \rad=\varrho_1N^{1/(d+2)}.
\end{equation}
Subsequently, more refined picture under $\mu_N$ has been proved in~\cite{Szn91a,Bol94,Pov99,DFSX18,BC18}: there exists a random center
\begin{equation}
\cent(\Oi)\in \ball{0}{\rad}
\end{equation}
such that for any $\epsilon>0$,
\begin{equation}
\lim_{N\to\infty}
\mu_N\left(\ball{\cent}{(1-\epsilon)\rad}\subset S_{[0,N]}\subset \ball{\cent}{(1+\epsilon)\rad}
\right)=1.
\label{eq:SBP}
\end{equation}
Note that the left inclusion in particular implies that the ball $B({\cent};{(1-\epsilon)\rad})$ is vacant.

The model with non-zero bias first appeared in the physics literature~\cite{GP82} where a phase transition of the asymptotic velocity was discussed. A rigorous proof of this ballisticity transition was given in~\cite{Szn95b1,Szn95b2}, as a consequence of a large deviation principle for $\mu_N(S_N/N \in \cdot )$. We shall provide a more detailed overview on related works in Section~\ref{sec:history}.

In this paper, we study the sub-ballistic phase of $\mu_N^h$ in detail. In order to state the results, we need to introduce the so-called Lyapunov exponent (or norm) which measures the cost for the random walk to make a long crossing among the obstacles. For 
$x=(x_1, \ldots, x_d)\in\R^d$, we write $[x]:=(\lfloor x_1\rfloor, \ldots, \lfloor x_d\rfloor)\in \Z^d$.

\begin{definition}
\label{def:beta}
The annealed Lyapunov exponent $\beta\colon\R^d\to [0,\infty)$ is defined by
\begin{equation}
 \beta(x)=-\lim_{n\to\infty}\frac1n \log \P\otimes \bP(\tau_\mathcal{O}>\tau_{[nx]}),
\label{eq:def_beta}
\end{equation}
and its dual norm $\beta^*$ is defined by
\begin{equation}
 \beta^*(h)=\sup\left\{\langle h, x\rangle \colon \beta(x)=1\right\}.
\end{equation}
\end{definition}
The set of critical points $h_c\in\R^d$ for the aforementioned ballisticity transition are characterized by this dual norm as $\beta^*(h_c)=1$. The existence of the limit in~\eqref{eq:def_beta} follows from the subadditive ergodic theorem. It can be further shown that
\begin{equation}
\label{eq:shape_thm}
 \lim_{|x|\to\infty}\frac{1}{|x|}|\beta(x)+\log \P\otimes \bP(\tau_\mathcal{O}>\tau_{x})|=0.
\end{equation}
See~\cite[Theorem~3.4 on p.244]{Szn98} for the corresponding result in the continuum setting.

Now we are ready to state the first main result of this paper, that is the large deviation principle under the annealed law without bias in the scale between $\rad$ and $o(\rad^d)$. For scales between $\rad^d$ to $N$, the large deviation principle is proved in~\cite{Szn95b1,Szn95b2}. We write $B(y;r)\subset\Z^d$ for the Euclidean ball centered at $y\in\R^d$ and radius $r>0$, and ${\rm dist}_\beta$ for the distance with respect to the Lyapunov norm $\beta(\cdot)$.
\begin{theorem}
\label{thm:LDP}
Let $d\ge 2$.
\begin{enumerate}
 \item  Let $\varphi(N)$ be such that $\rad \ll \varphi(N) \ll \rad^d$. Then for any $x\in\R^d$,
\begin{equation}
 \mu_N(S_N=[\varphi(N) x])=\exp\{-\beta(x)\varphi(N)(1+o(1))\},
\label{eq:LDP>>rad}
\end{equation}
as $N\to\infty$.
\item For any $x\in\R^d$,
\begin{equation}
 \mu_N(S_N=[\rad x])=\exp\left\{-{\rm dist}_\beta(x,\ball{0}{2})\rad(1+o(1))\right\},
\label{eq:LDP=rad}
\end{equation}
as $N\to\infty$.
\end{enumerate}
\end{theorem}
The form of the rate functions reflects the following facts. First, we can let the random walk reach any point $y\in\ball{0}{2\rad}$ with a negligible cost by shifting the center $\cent$ of the vacant ball in~\eqref{eq:SBP} so that $\ball{\cent}{\rad}$ contains $0$ and $y$. This is why the rate function is zero inside $\ball{0}{2}$ in~\eqref{eq:LDP=rad}. Next, when $[\varphi(N) x]\not\in\ball{0}{2\rad}$, it turns out that the best strategy is still to have a vacant ball of radius almost $\rad$. Thus the cost for the random walk to reach $[\varphi(N) x]$ comes solely from the crossing from $\ball{0}{2\rad}$ to $[\varphi(N) x]$, and it is measured by the Lyapunov norm $\beta$. This explains the form of rate function in~\eqref{eq:LDP=rad}. In~\eqref{eq:LDP>>rad}, the size of $\ball{0}{2\rad}$ is negligible compared with $\varphi(N)$ and hence it does not affect the asymptotics.

The second main result in this paper is a detailed description of the behavior of the random walk under $\mu_N^h$ with a sub-critical drift. As $\mu_N^h$ is obtained by tilting $\mu_N$ by $e^{\langle h,S_N\rangle}$, the competition between the gain $\langle h,S_N\rangle$ and the cost for the displacement in Theorem~\ref{thm:LDP} determines the behavior of $S_N$. The following theorem describes not only the endpoint but also the whole path behavior.
\begin{theorem}
\label{thm:confine} Let $d\geq 2$. 
Suppose $\beta^*(h) < 1$. Then for any $\epsilon>0$,
\begin{equation}
\lim_{N\to\infty} \mu^h_N\left(
\ball{\rad\be_h}{(1-\epsilon)\rad}
\subset
S_{[0,N]}\subset \ball{\rad\be_h}{(1+\epsilon)\rad}
\right)=1,
\label{eq:confine}
\end{equation}
where $\be_h:=h/|h|$. Furthermore,
\begin{equation}
\label{eq:LLN}
\lim_{N\to\infty}\frac{1}{\rad}S_N=2\be_h
\text{ in $ \mu^h_N$-probability,}
\end{equation}
\end{theorem}
\begin{remark}
With a little more effort, we can replace $\epsilon$ in the above theorem by $\rad^{-c}$ for a small $c>0$. However, since our argument does not seem to give a good control on $c$, we decided not to present the proof of this refinement.
\end{remark}
\begin{remark}
Our argument for Theorem~\ref{thm:confine} provides a proof of~\eqref{eq:confine} in the case $h=0$ as well. See Remarks~\ref{rem:unbiased1} and~\ref{rem:unbiased2}. In this case, it can be regarded as a combination of the ideas from~\cite{Szn91b,Pov99} and from~\cite{Bol94}.
\end{remark}
The first assertion~\eqref{eq:confine} says that the result~\eqref{eq:SBP} remains true under $\mu_N^h$ if $\beta^*(h)<1$ but the center $\cent$ of the ball becomes $\rad\be_h$. The second assertion~\eqref{eq:LLN} is natural since this strategy maximizes the weight $e^{{\langle h, S_N\rangle}}$ in~\eqref{eq:biased_law} under the constraint in~\eqref{eq:confine}.
\subsection{Related works}
\label{sec:history}
We give a brief overview on the earlier works related to our results. The problem of diffusing particle among the traps has been discussed in the continuum and discrete settings in parallel and most of the results hold in both cases without change. For this reason, we often refrain from indicating in which setting the results are proved.

This type of model with non-zero bias first appeared in the physics literature~\cite{GP82} where a ballisticity transition was discussed. On the mathematical side, the first result seems to be~\cite{EL87} where a phase transition for the free energy of $\mu_N^h$ is proved. In particular, it is proved that when $d\ge 2$ and the bias is small, the partition function of $\mu_N^h$ has the same asymptotics as~\eqref{eq:DV79}.
Then in 1990s, Sznitman studied this model and its quenched counterpart in a series of works. We summarize some of the results related to this paper. In~\cite{Szn95b1,Szn95b2}, the annealed Lyapunov exponent with an additional parameter $\lambda\ge 0$ was defined as follows:
\begin{equation}
 \beta_\lambda(x)=-\lim_{n\to\infty}\frac1n \log \E\otimes \bE\left[e^{-\lambda\tau_{[nx]}}\colon\tau_\mathcal{O}>\tau_{[nx]}\right].
\end{equation}
Then, improving upon earlier results in~\cite{Szn91a}, it is proved for $\rad^d\le \varphi(N)\le N$ that the law of $S_N/\varphi(N)$ under $\mu_N$ satisfies a large deviation principle at rate $\varphi(N)$ with rate function
\begin{equation*}
\begin{cases}
   J(x)=\sup\{\beta_\lambda(x)-\lambda\colon \lambda\ge 0\}
&\text{if }\varphi(N)=N \text{ (\cite[(0.2), (0.3)]{Szn95b1})}, \\
 \beta(x)
&\text{if }\rad^d \ll \varphi(N) \ll N \text{ (\cite[(0.4)]{Szn95b1})},\\
{\beta(x)}
&{\text{if } \varphi(N)=\rad^d \text{ and } d\ge 2\text{ (\cite[(0.4)]{Szn95b2})}}.
\end{cases}
\end{equation*}
We discuss the case $d=1$ and $\varphi(N)=\rad^d$ later.

By a standard tilting argument (see~\cite[Theorem~II.7.2]{Ell85}, for example), the above large deviation principle can be transferred to those for $\mu_N^h(S_N/\varphi(N)\in\cdot)$ at the same rate with rate function
\begin{equation*}
\begin{cases}
J^h(x)=J(x)-\langle h,x \rangle -\inf_{y\in\R^d}\{J(y)-\langle h,y \rangle\}
&\text{if }\varphi(N)=N \text{ (\cite[Theorem~2.1]{Szn95b2})},\\
\beta(x)-\langle h,x \rangle
&\text{if }\beta^*(h)<1 \text{ and }\rad^d \ll \varphi(N) \ll N,\\
\beta(x)-\langle h,x \rangle
&\text{if }\beta^*(h)<1, d\ge 2 \text{ and }\varphi(N)=\rad^d\\
&\text{ (\cite[Theorem~2.2]{Szn95b2})}.
\end{cases}
\end{equation*}
The first one in particular implies the phase transition of the velocity at $\beta^*(h)=1$ (see~\cite[Corollary~4.10 on p.262]{Szn98} for the precise statement). The third one implies that in the subcritical phase, the endpoint of the walk is of distance $o(\rad^d)$ from the origin. This also extends the result of~\cite{EL87} to the whole subcritical phase. See also~\cite{Flu07,Flu08} for the results in the discrete setting.
Later Ioffe and Velenik studied the ballistic phase in more detail. An interested reader is referred to~\cite{Iof15}. Among other things, it is proved in~\cite{IV12} that the walk is ballistic at criticality. Thus what has been left open is the precise scaling limit under the subcritical drift. Theorem~\ref{thm:confine} fills this missing piece and completes the picture.

Let us also mention that more is known in dimension one. The ballisticity transition follows from the results in~\cite{Szn95b1}. The results corresponding to Theorems~\ref{thm:LDP} and~\ref{thm:confine} are proved in~\cite{Pov95} and~\cite{Pov97}, respectively, but with some notable differences. First, unlike in our Theorem~\ref{thm:LDP}, the rate function in~\cite{Pov97} in the scale $\rad$ has a vanishing gradient at $|x|=2$. The reason for this singularity in $d=1$ is that the costs for $\tau_\Oi>N$ and $S_N=[\rad x]$ cannot be separated. In order for the random walk to reach $[\rad x]$, the interval $[0,[\rad x]]$ must be free of obstacles and that certainly helps to have $\tau_\Oi>N$. When $d\ge 2$, the size of the vacant ball is essentially determined by the leading term in~\eqref{eq:DV79} which is much larger than $\rad$, and hence we can separate the cost for $S_N=[\rad x]$ as we explained after Theorem~\ref{thm:LDP}. Second, in~\cite{Pov95}, the path behavior is studied not only on the macroscopic scale $\rad$ but also on the microscopic scale $O(1)$. On the latter scale, the result roughly says that the walk behaves as if it is conditioned to stay away from a wall with a random position, which lies at the first obstacle to the left of the origin. Though the macroscopic scaling result was later extended to the so-called ``soft obstacles'' in~\cite{Set03}, the microscopic scaling problem remains open in that case. Finally around the critical bias, the asymptotic speed is proved to be continuous in the hard obstacles case in~\cite[Theorem~5.1]{IV10}, whereas~\cite[Corollary~1.1]{KM12} implies that it is discontinuous in the case of soft obstacles. Later in~\cite{KS17}, it is proved that the walk with the critical bias among hard obstacles scales like $N^{1/2}$.
\medskip

\begin{remark}
The behavior on the microscopic scale in higher dimensions is a very interesting open problem. The difficulty in the soft obstacles and higher dimensional cases is that, unlike in the case of one-dimensional hard obstacles, a single obstacle cannot play the role of a wall. One needs to understand the geometry of the obstacles configuration around the starting point of the random walk under the effect of the conditioning on the long time survival.
\end{remark}

\section{Outline of proofs}
In this section, we explain the outline of the proofs and the organization of the rest of this paper. The main conceptual difficulty is that we are studying events whose probability decay much slower than the partition functions. This is particularly easy to see in Theorem~\ref{thm:LDP}: from~\eqref{eq:DV79}, we know the asymptotics of the partition function
\begin{equation}
 \P\otimes\bP\left(\tau_\Oi>N\right)=\exp\left\{-c(d,p)N^{\frac{d}{d+2}}+o\left(N^{\frac{d}{d+2}}\right)\right\},
\label{eq:DV}
\end{equation}
but the error term is much larger than the leading terms in Theorem~\ref{thm:LDP}.
Since we have little further information about the error term, it is difficult to prove Theorem~\ref{thm:LDP} by computing the asymptotics of $\P\otimes\bP(\tau_\Oi>N, S_N=x)$ explicitly. Instead, we will use comparison arguments to say something about the path measure, comparing different strategies to achieve $\{\tau_\Oi>N, S_N=x\}$. Roughly speaking, it turns out that when $|x|=o(\rad^d)$, one of the best strategies for the random walk, which gives the dominant contribution, is to stay in a vacant ball of radius almost $\rad$ for a long time and then go to $x$ during the small time interval near the end. As a result, the costs for surviving for a long time and reaching $x$ without hitting the obstacles in $\P\otimes\bP(\tau_\Oi>N, S_N=x)$ can be separated. The former cost counterbalance the partition function and the latter cost gives the rate function in Theorem~\ref{thm:LDP}. For a technical reason, to be explained in Remark~\ref{rem:switch}, we will work under a slightly different conditioning: Let $\tau_x^N$ be the first hitting time of $x$ after time $N$ and define
\begin{equation}
\mu_{N,x}(\cdot) =\P\otimes \bP(\cdot\mid \tau_\Oi> \tau_x^N).
\end{equation}

Now let us describe in more detail how the rest of the paper is organized.

In Section~\ref{sec:LDPlower}, we show the lower bound in Theorem~\ref{thm:LDP}. In particular, it implies a lower bound on the partition function of $\mu_{N,x}$ since $\{S_N=x,\tau_\Oi>N\}\subset\{\tau_\Oi>\tau_x^N\}$. The proof is based on the construction of a specific strategy to achieve $S_N=x$ and $\tau_\Oi>N$. We first use~\eqref{eq:SBP} to find a vacant (i.e., free of obstacles) ball ``shifted toward $x$'' and let the random walk stay there most of the time. Then in the final $O({\rm dist}_\beta(x,\ball{0}{2\rad}))$ time, we let the random walk go to $x$. The first part has a probability comparable to $\P\otimes\bP(\tau_\Oi>N)$, while the probability of the second part decays exponentially in ${\rm dist}_\beta(x,2\rad)$. We use the FKG inequality to separate these two parts.

In Section~\ref{sec:vacant}, we show that under $\mu_{N,x}$ with $|x|=o(\rad^d)$, there exists a vacant ball of radius almost $\rad$, just as in~\eqref{eq:SBP}. We also show that it is hard for the random walk to survive outside the vacant ball. For the proofs, we will first use a coarse graining scheme from~\cite{DX18} (or alternatively the method of enlargement of obstacles in~\cite{Szn97a}) to show that there exists an almost vacant ball. Then we use a density dichotomy lemma from~\cite{DFSX18} to conclude that the ball is completely vacant.

In Section~\ref{sec:outsideB}, we show that the random walk under $\mu_{N,x}$ with $|x|=o(\rad^d)$ will spend only little time outside the vacant ball. More precisely, we first prove that the time spent before reaching and after leaving the vacant ball cannot be too long. Second, we prove that the random walk path between the first and last visit to the vacant ball is confined in a slightly larger and concentric ball. The proofs rely on the results in Section~\ref{sec:vacant} and a path switching argument in the same spirit as in~\cite{DFSX18}.

In Section~\ref{sec:cost=dist}, we essentially show that the cost for $\tau_\Oi>\tau_x^N$ can be separated into three parts: (i) crossings from the origin to the vacant ball, (ii) staying near the vacant ball, and (iii) crossing from the vacant ball to $x$. Due to the confinement proved in Section~\ref{sec:outsideB}, part (ii) is independent from other parts and has probability comparable to $\P\otimes\bP(\tau_\Oi>N)$. If parts (i) and (iii) are nearly independent, then the costs are measured by the distances from the origin and $x$ to the vacant ball with respect to the Lyapunov norm, respectively. As it is not easy to control the dependence between (i) and (iii), we will modify them in the proof. See Proposition~\ref{prop:z_fixed} for the precise formulation. Adapting the same argument, we also prove that when $|x|$ is close to $2\rad$, then the whole random walk path is confined in a small neighborhood of the vacant ball under $\mu_{N,x}$.

In Section~\ref{sec:LDPupper}, we prove the upper bound in Theorem~\ref{thm:LDP}. This follows almost directly from the first result in Section~\ref{sec:cost=dist} since $\{S_N=x,\tau_\Oi>N\}\subset\{\tau_\Oi>\tau_x^N\}$.

In Section~\ref{sec:thm_drift}, we prove Theorem~\ref{thm:confine}. The law of large numbers~\eqref{eq:LLN} can be deduced from Theorem~\ref{thm:LDP} and large deviation results in~\cite{Szn95b1,Szn95b2} via a standard tilting argument, but we will present a more direct argument.
In order to prove the confinement~\eqref{eq:confine}, we use~\eqref{eq:LLN} to relate $\mu_N^h$ to the random walk law conditioned to avoid obstacles up to time $N$ and end around $2\rad\be_h$. By using the results in Section~\ref{sec:outsideB}, this latter law can further be related to $\mu_{N,x}$ with $x$ close to $2\rad\be_h$, for which the confinement is proved in Section~\ref{sec:cost=dist}.


\section{Notation and preliminaries}
We will prove various intermediate propositions with error terms depending on $|x|$. To simplify the notation, we define
\begin{equation}
\delta_{N,x}={\rad^{-1/5}\vee ({|x|}/{\rad^{d}}),}
\label{eq:def_delta}
\end{equation}
which goes to zero polynomially fast in $N$ when $|x|\le \rad^{d-\xi}$ for some $\xi>0$. The exponent $-1/5$ is rather arbitrary and has no significance.

We use $c$ and $c'$ to denote a positive constant whose value may change from line to line. When we need to keep the value of a constant within a proof, we use the upper case letters $C$, $C_1$ and $C_2$. We write $c_{X.Y}$ for a constant defined in Theorem/Proposition/Lemma~$X.Y$, if it is referred to in other places.

Next, we collect some notation and estimates for the simple symmetric random walk on $\Z^d$. For $U\subset\Z^d$, we denote by $\lambda_U$ the smallest Dirichlet eigenvalue of the discrete Laplacian $-\frac1{2d}\Delta$. Then, we have the following tail estimate for the exit time $\tau_{U^c}$ from $U$:
\begin{equation}
\begin{split}
\bP_x(\tau_{U^c}>n) &\le |U|^{1/2}(1-\lambda_U)^n\\
&\le |U|^{1/2}\exp\left\{-n\lambda_U\right\}.
\end{split}
\label{eq:SG}
\end{equation}
See~\cite[(2.21)]{Kon16} for the continuous time analogue. A similar bound with $|U|^{1/2}$ replaced by $c(1+n\lambda_U)^{d/2}$ also holds, see~\cite[(1.9) in Section~3]{Szn98}. Combining this with a Faber--Krahn type inequality $\lambda_U \ge c|U|^{-2/d}$ for the eigenvalue of Laplacian on $\Z^d$~\cite[Remark~3.2.6 and Theorem~3.2.7]{Kum14}, we can deduce that
\begin{equation}
\bP_x(\tau_{U^c}>n) \le c\exp\left\{-c n|U|^{-2/d}\right\}.
\label{eq:SG2}
\end{equation}

For $U\subset\Z^d$, $n\ge 0$ and $x,y\in U$, we write $p^U_n(x,y)$ for the transition probability of the random walk killed upon exiting from $U$. Whenever we use this notation, we tacitly assume that $|y|_1$ has the same parity as $n+|x|_1$.
\begin{lemma}
There exists $c>0$ such that for any $R\ge 2$, $n\ge R^2/2$ and $x,y\in\ball{0}{R}$,
\begin{equation}
p^{\ball{0}{R}}_n(x,y) \ge \frac{c}{R^{d}}
d_R(x)d_R(y)
\exp\left\{-\frac{n}{cR^2}\right\},
\label{eq:ballHK}
\end{equation}
where $d_R(z)=R^{-1}{\rm dist}_{\ell^1}(z,\partial{\ball{0}{R}})$.
If $x=y\in\ball{0}{R}$, the same bound holds for all $n\ge 0$.
\end{lemma}
\begin{proof}
When $n\in [R^2/2,R^2]$, this is a consequence of~\cite[Proposition~6.9.4]{LL10}. For $n\ge R^2$, we use the Chapman--Kolmogorov identity to obtain
\begin{equation}
p^{\ball{0}{R}}_n(x,y)
\ge \sum_{z\in\ball{0}{R/2}}p_{R^2/2}^{\ball{0}{R}}(x,z)p_{n-R^2/2}^{\ball{0}{R}}(z,y).
\end{equation}
The second factor is bounded from below by $cd_R(y)R^{-d}\exp\{-c^{-1}nR^{-2}\}$, uniformly in $z\in\ball{0}{R/2}$, by the second part of~\cite[Proposition~6.9.4]{LL10} and~\cite[Corollaries~6.9.5 and~6.9.6]{LL10}.
Then, we use the result for $n=R^2/2$ to obtain
\begin{equation}
 \sum_{z\in\ball{0}{R/2}}p^{\ball{0}{R}}(x,z)\ge c d_R(x).
\end{equation}
Combining the above two estimates, we obtain~\eqref{eq:ballHK}.

Next, let $x=y\in\ball{0}{R}$, which forces $n\in 2\N$. For $n \in\{0, 2\}$, the left-hand side of~\eqref{eq:ballHK} is larger than $(2d)^{-2}$. For $n\in [4,R^2/2]$, we can find a ball $\ball{z}{\sqrt{n}}$ such that ${\rm dist}_{\ell^1}(x,\partial{\ball{z}{\sqrt{n}}})={\rm dist}_{\ell^1}(x,\partial{\ball{0}{R}})$. Applying the first part of~\cite[Proposition~6.9.4]{LL10} to this ball, we obtain~\eqref{eq:ballHK} in this case.
\end{proof}
\section{Proof of the lower bound in Theorem~\ref{thm:LDP}}
\label{sec:LDPlower}
We first show a lower bound on the survival probability with a fixed endpoint, which in particular implies the lower bounds in Theorem~\ref{thm:LDP}.
\begin{proposition}
\label{prop:LDPlower}
There exists $c_{\text{\upshape\ref{prop:LDPlower}}}>0$ such that when $\epsilon>0$ is small depending on $d$ and $p$ and $|x|\le \epsilon\rad^d$,
\begin{equation}
\P\otimes \bP(\tau_\Oi>N, S_N=x)
\ge \exp\left\{-{\rm dist}_\beta\left(x,\ball{0}{2\rad}\right)-c_{\text{\upshape\ref{prop:LDPlower}}}\epsilon(|x|\vee \rad)\right\}\P\otimes \bP(\tau_\Oi>N)
\label{eq:LDPlower}
\end{equation}
for all sufficiently large $N$. Furthermore, under the same condition,
\begin{equation}
\begin{split}
 \P\otimes \bP\left(\tau_\Oi>\tau_x^N\right)
& \ge \exp\left\{-c(d,p)N^{\frac{d}{d+2}}-{\rm dist}_\beta\left(x,\ball{0}{2\rad}\right)-{c_{\text{\upshape\ref{prop:LDPlower}}}} \delta_{N,x}N^{\frac{d}{d+2}}\right\}
\end{split}
\label{eq:pf}
\end{equation}
for all sufficiently large $N$, where $c(d,p)$ and $\delta_{N,x}$ are defined in~\eqref{eq:DV} and~\eqref{eq:def_delta}, respectively.
\end{proposition}
\begin{proof}
\begin{figure}
 \centering
 \includegraphics{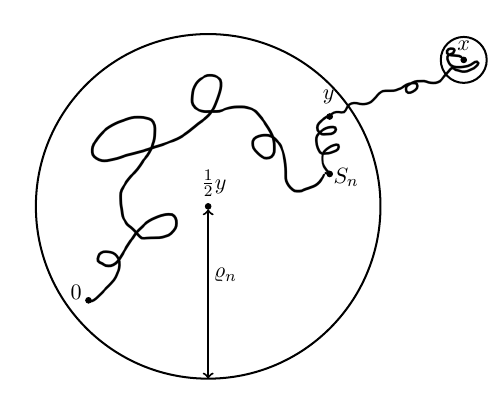}
\caption{The strategy to achieve the large deviation lower bound in Proposition~\ref{prop:LDPlower}. Since $\lim_{N\to\infty}N/n=1$ for $n$ defined in~\eqref{eq:def_n}, the large ball has radius almost $\rad$.
\label{fig:LDPlower}}
\end{figure}
We start by introducing several objects used in this proof. Let us first assume $|x|\ge 2\rad$ and let $y\in\ball{0}{(2-4\epsilon)\rad}$ be such that $\beta(x-y)={\rm dist}_\beta\left(x,\ball{0}{(2-4\epsilon)\rad}\right)$. Then for $M>0$ to be chosen later in~\eqref{eq:(iii)} depending only on $d$ and $p$, define
\begin{equation}
 n=N-M|x-y|-\rad^2\ge N-2M\rad^d.
\label{eq:def_n}
\end{equation}
Roughly speaking, we consider the following strategy: There is a ball of radius $\varrho_n$ centered around $\tfrac12 y$ which is free of obstacles, and we let the random walk (i) stay inside that ball up to time $n$, (ii) get close to $y$ in the next $\rad^2$ steps, (iii) go to $x$ in the remaining $M|x-y|$ steps (See Figure~\ref{fig:LDPlower}). The cost up to (ii) is comparable to $\P\otimes\bP(\tau_\Oi>N)$ while the cost for (iii) is measured by $\exp\{-\beta(x-y)\}$. The following argument makes this outline rigorous.

It is proved in~\cite{DFSX18} that for $\mathcal{x}_n$ in~\eqref{eq:SBP} and any $\epsilon>0$,
\begin{equation}
 \P\otimes\bP\left(
\Oi\cap\ball{\mathcal{x}_n}{(1-\epsilon)\varrho_n}=\emptyset\cond \tau_\Oi>n\right)\to 1
\end{equation}
as $n\to\infty$. Moreover, we know from~\cite{Szn91b,Pov99} that the distribution of $\varrho_n^{-1}\mathcal{x}_n$ converges to $\phi_1(x)\dd x$, where $\phi_1$ is the $L^1$-normalized principal eigenfunction of the Dirichlet Laplacian in $\ball{0}{1}\subset\R^d$. Since $\phi_1$ is positive and continuous inside $\ball{0}{1}$, there exists $c(\epsilon)>0$ such that
\begin{equation}
\P\otimes\bP\left(\mathcal{x}_n \in \ball{\tfrac12 y}{\epsilon\varrho_n}\cond \tau_\Oi>n\right) \ge c(\epsilon)
\end{equation}
for all $n\ge 1$. Recall also that~\cite[Lemma~4.5]{DFSX18} shows
\begin{equation}
 \P\otimes\bP\left(S_n\in \ball{\mathcal{x}_n}{(1-4\epsilon)\varrho_n}\mid \tau_\Oi>n\right)\to 1
\end{equation}
as $n\to\infty$ and $\epsilon\to 0$.
Summarizing the above considerations, when $\epsilon>0$ is sufficiently small, we have
\begin{equation}
\begin{split}
&\P\otimes\bP\left(S_n\in \ball{\tfrac12 y}{(1-3\epsilon)\varrho_n},  \Oi\cap\ball{\tfrac12 y}{(1-2\epsilon)\varrho_n}=\emptyset,\tau_\Oi>n\right)\\
&\quad \ge \frac{c(\epsilon)}{2} \P\otimes\bP\left(\tau_\Oi>n\right)
\end{split}
\end{equation}
for all sufficiently large $n$.
On the event on the left-hand side, we further let the random walk go to $y$ inside $\ball{\tfrac12 y}{(1-2\epsilon)\varrho_n}$ during the time interval $[n,n+\rad^2]$. Then using the Markov property at time $n$ and the random walk estimate~\eqref{eq:ballHK}, we obtain
\begin{equation}
\begin{split}
\P\otimes\bP\left(S_{n+\rad^2}= y, \tau_\Oi>n+\rad^2\right)
\ge \frac{c(\epsilon)}{\varrho_n^d} \P\otimes\bP\left(\tau_\Oi>n\right)
\end{split}
\label{eq:upto(ii)}
\end{equation}
for all sufficiently large $n$.

Next we let the random walk go from $y$ to $x$ during the time interval $[n+\rad^2, N]$ without hitting the obstacles. By imposing an extra condition $\Oi\cap\ball{x}{R}=\emptyset$ for $R=|x-y|^{1/2d}$, the probability of this last piece is bounded from below by
\begin{equation}
\begin{split}
& \P\otimes\bP_{y}\left(S_{M|x-y|}=x, \tau_\Oi>M|x-y|\right)\\
&\quad \ge \E\otimes\bE_{y}\left[p_{M|x-y|-\tau_x}^{\ball{x}{R}}(x,x) \colon \Oi\cap\ball{x}{R}=\emptyset, \tau_x<\tau_\Oi\wedge {M}|x-y|\right]\\
&\quad \ge p^{|\ball{x}{R}|}\min_{{k\in 2\Z: {0}\le k\le M|x-y|}}p_k^{\ball{x}{R}}(x,x)\P\otimes\bP_{y}\left(\tau_x<\tau_\Oi\wedge M|x-y|\right),
\end{split}
\label{eq:hitting_x}
\end{equation}
where we have applied the FKG inequality to $1_{\{\Oi\cap\ball{x}{R}=\emptyset\}}$ and $\bP_{y}(\tau_x<\tau_\Oi\wedge M|x-y|)$,
which are decreasing functions in the obstacle field. 
Due to the random walk estimate~\eqref{eq:ballHK}, the above $p_{k}^{\ball{x}{R}}(x,x)$ is bounded from below by $cR^{-d}\exp\{-c^{-1} M|x-y|/R^2\}$. Since we have chosen $R=|x-y|^{1/2d}$, it follows that
\begin{equation}
p^{|\ball{x}{R}|}\min_{k\in2\Z: k\le M|x-y|}p_k^{\ball{x}{R}}(x,x)
\ge \exp\{-c M|x-y|^{1-1/d}\}.
\end{equation}
To bound the third factor in the third line of~\eqref{eq:hitting_x}, we use a result in~\cite[Theorem~1.1]{KM12} which says that $\E\otimes\bE[\tau_z\mid \tau_\Oi>\tau_z]\le C|z|$. From this and the Markov inequality, it follows that
\begin{equation}
\begin{split}
 \P\otimes\bP_{y}\left(\tau_x<\tau_\Oi\wedge {M}|x-y|\right)
 & \ge \frac12\P\otimes \bP_{y}(\tau_\Oi > \tau_x)\\
 & \ge \exp\{-(1+\epsilon)\beta(x-y)\}
\end{split}
\label{eq:(iii)}
\end{equation}
when $\epsilon$ is small and $M, N$ are large, where in the second inequality we have used~\eqref{eq:shape_thm} and that $|x-y| \ge c\epsilon\rad$ for $|x|\ge 2\rad$.

Finally, since $\bP(S_{n+\rad^2}= y, \tau_\Oi>n)$ and $\bP_{{y}}(S_N=x, \tau_\Oi>M|x-y|)$ are both decreasing in $\Oi$, we can use the FKG inequality to deduce from~\eqref{eq:upto(ii)}--\eqref{eq:(iii)} that
\begin{equation}
\begin{split}
&\P\otimes\bP\left(S_N=x, \tau_\Oi>N\right)\\
&\quad\ge \E\left[\bP(S_{n+\rad^2}= y, \tau_\Oi>n+\rad^2)\bP_{y}(S_{M|x-y|}=x, \tau_\Oi>M|x-y|)\right]\\
&\quad\ge \exp\left\{-(1+\epsilon)\beta(x-y)-c\epsilon(|x|\vee\rad)\right\}\P\otimes\bP\left(\tau_\Oi>n\right)
\end{split}
\end{equation}
for all sufficiently large $N$. Since $\beta(x-y)\le c|x|$, this concludes the proof of~\eqref{eq:LDPlower} in the case $|x|\ge 2\rad$.

Let us turn to the case $|x|<2\rad$. If we assume a slightly stronger condition $|x|\le (2-4\epsilon)\rad$, then we have $y=x$ and $n=N-\rad^2$ and hence~\eqref{eq:upto(ii)} gives us the desired bound. If $(2-4\epsilon)\rad \le |x|\le 2\rad$, then we set $y$ as before and let $n=N-|y-x|_1$. We follow the same argument up to~\eqref{eq:upto(ii)}. Then instead of~\eqref{eq:(iii)}, we fix a path $\pi(y,x)$ connecting $y$ and $x$ with $|y-x|_1$ steps and use
\begin{equation}
\begin{split}
\P\otimes\bP_y\left(S_{|y-x|_1}=x, \Oi\cap \pi(y,x)=\emptyset\right)
&=\left(\frac{p}{2d}\right)^{|y-x|_1}\\
&\ge \exp\left\{-c\epsilon \rad\right\}.
\end{split}
\end{equation}
Then following the same argument as above, we obtain~\eqref{eq:LDPlower} in this case.

The second assertion~\eqref{eq:pf} follows from~\eqref{eq:LDPlower} and the bound
\begin{equation}
 \P\otimes\bP(\tau_\Oi>N)\ge \exp\left\{-c(d,p)N^{\frac{d}{d+2}}-cN^{\frac{d-1}{d+2}}\right\}
\end{equation}
proved in~\cite[Proposition~2.1]{Bol94} by making $c_{\text{\upshape\ref{prop:LDPlower}}}$ larger.
\end{proof}
As a consequence of Proposition~\ref{prop:LDPlower}, we have a crude upper bound on $\tau^N_x$, the first hitting time of $x$ after $N$:
\begin{corollary}
\label{cor:toolong}
There exists $c_{\ref{cor:toolong}}>0$ such that when $\epsilon>0$ is small depending on $d$ and $p$ and $|x|\le \epsilon\rad^d$,
\begin{equation}
\mu_{N,x}\left(\tau_x^N> 2N\right) \le \exp\left\{-c_{\ref{cor:toolong}}N^{\frac{d}{d+2}}\right\}
\label{eq:toolong}
\end{equation}
for all sufficiently large $N$.
\end{corollary}
\begin{proof}
By~\eqref{eq:DV79},
\begin{equation}
\begin{split}
\P\otimes \bP\left(\tau_\Oi>\tau^N_x>2N\right)
&\le \P\otimes \bP\left(\tau_\Oi>2N\right)\\
& \le  \exp\left\{-(c(d,p)+o(1))(2N)^{\frac{d}{d+2}}\right\},
\end{split}
\end{equation}
as $N\to\infty$. Comparing this with~\eqref{eq:pf}, we get~\eqref{eq:toolong}.
\end{proof}
\begin{remark}
\label{rem:macroball}
Due to Corollary~\ref{cor:toolong}, we may effectively discard the event $\{\tau_x^N>2N\}$ from our consideration. Thus in what follows, we will tacitly assume $\tau_x^N\le 2N$. Since we are considering the discrete time random walk, this in particular implies that all the points of $\Z^d$ appearing hereafter can be assumed to be in $\ball{0}{2N}$. In particular, we will replace the set of obstacles $\Oi$ by $\Oi\cup \ball{0}{2N}^c$.
\end{remark}

\section{Existence of a vacant ball}
\label{sec:vacant}
The main result in this section is the existence of a ball of radius almost $\rad$ which is free of obstacles under the measure $\mu_{N,x}$ with $|x|=o(\rad^d)$.
\begin{proposition}
\label{prop:vacant}
There exist $\cent(\Oi)\in\ball{0}{\rad}$ and $c_{\ref{eq:vacant}}>0$ such that when $\epsilon>0$ is small depending on $d$ and $p$ and $|x|\le \epsilon\rad^d$, the $\mu_{N,x}$-probability of the events
\begin{equation}
\label{eq:vacant}
\left\{\Oi\cap \ball{\cent}{(1-\delta_{N,x}^{c_{\ref{eq:vacant}}})\rad}=\emptyset\right\}
\end{equation}
 and
\begin{equation}
\label{eq:covering}
\left\{\ball{\cent}{(1-\delta_{N,x}^{c_{\ref{eq:vacant}}})\rad}\subset S_{[0,\tau_x^N]}\right\}
\end{equation}
are greater than $1-\exp\{-(\log N)^2\}$ for all sufficiently large $N$, where $\delta_{N,x}$ is defined in~\eqref{eq:def_delta}.
\end{proposition}
We deduce Proposition~\ref{prop:vacant} from the following two lemmas. The first one asserts that there is a ball of radius $\rad$ which is almost free of obstacles; the second one asserts that every obstacle is well surrounded by others.
\begin{lemma}
\label{lem:almost}
There exists $c_{\ref{lem:almost}}>0$ and $\cent(\Oi)\in\ball{0}{\rad}$ such that when $\epsilon>0$ is small depending on $d$ and $p$ and $|x|\le \epsilon\rad^d$,
\begin{equation}
\mu_{N,x}\left(|\Oi\cap \ball{\cent}{\rad}|
 \ge \delta_{N,x}^{c_{\ref{lem:almost}}}N^{\frac{d}{d+2}}\right)
\le \exp\left\{-{c_{\ref{lem:almost}}} \delta_{N,x} N^{\frac{d}{d+2}}\right\}
\label{eq:almost}
\end{equation}
for all sufficiently large $N$.
\end{lemma}
\begin{lemma}
\label{lem:odensity} For each $v\in \Z^d$, $l>0$ and $\delta>0$, let
 \begin{equation}
 E_l^\delta(v)=\left\{v\in \Oi \textrm{ and }{\frac{|\Oi \cap \ball{v}{l}|}{|\ball{v}{l}|}} < \delta \right\}.
 \end{equation}
Then there exists $c_{\ref{lem:odensity}}>0$ such that for sufficiently large $N$,
\begin{equation}
 \label{odensity}
 \mu_{N,x}\left(\bigcup_{v\in \ball{0}{2N}} \, \bigcup_{{(\log N)^3}\leq l\leq {\rad}} E_l^{c_{\ref{lem:odensity}}}(v)\right) \leq \exp\left\{-c_{\ref{lem:odensity}} (\log N)^3\right\}.
\end{equation}
\end{lemma}
\begin{proof}
[Proof of Proposition~\ref{prop:vacant}]
If there is an obstacle deep inside the ball $B(\cent;\rad)$ found in Lemma~\ref{lem:almost}, then there are in fact many obstacles by Lemma~\ref{lem:odensity}, which contradicts~\eqref{eq:almost}. The other part~\eqref{eq:covering} can also be deduced from these lemmas in the same way as~\cite[Lemma~3.2]{DFSX18}.
\end{proof}
In the proof of~\eqref{eq:confine}, we need to know that it is hard for the random walk to stay outside the vacant ball. The next lemma gives us such an estimate. For technical reasons, we will consider a slightly smaller ball
\begin{equation}
\begin{split}
 B^-(z)&=\ball{z}{(1-2\delta_{N,x}^{c_{\ref{prop:vacant}}})\rad},\\
 B^-&=B^-(\cent).
\end{split}
\label{eq:def_B-}
\end{equation}
Recall from Remark~\ref{rem:macroball} that we enlarged the obstacles to $\Oi\cup\ball{0}{2N}^c$.
\begin{lemma}
\label{lem:slow_cross}
There exists $c_{\text{\upshape{\ref{lem:slow_cross}}}}>0$ such that when $\epsilon>0$ is small depending on $d$ and $p$ and $|x|\le \epsilon\rad^d$, for $t\ge \delta_{N,x}^{c_{\text{\upshape{\ref{lem:slow_cross}}}}}(\log N)^2\rad^2$, $\mu_{N,x}$-probability of the event (which depends only on $\Oi$)
\begin{equation}
\label{eq:slow_cross}
\left\{{\sup_{y\in\ball{0}{2N}}}\bP_y\left(S_{[0,t]}\cap \left({\Oi} \cup B^-\right)=\emptyset\right) \leq \exp\left\{-\delta_{N,x}^{-c_{\text{\upshape{\ref{lem:slow_cross}}}}} \rad^{-2} t\right\}\right\}
\end{equation}
is greater than $1-\exp\{-(\log N)^2\}$ for all sufficiently large $N$.
\end{lemma}

The proofs of Lemmas~\ref{lem:almost} and~\ref{lem:slow_cross} are given in Section~\ref{appendix}. In Appendix~A, we shall also provide alternative proofs by adapting the arguments in~\cite{Pov99}, which are based on the so-called method of enlargement of obstacles in~\cite{Szn97a}.

Lemma~\ref{lem:odensity} is an analogue of~\cite[Lemma~2.1]{DFSX18}. We will provide an outline of the argument in Section~\ref{sec:odensity}.

\subsection{Proofs of Lemmas~\ref{lem:almost} and~\ref{lem:slow_cross}}
\label{appendix}
We prove Lemmas~\ref{lem:almost} and~\ref{lem:slow_cross} using some concepts and results from the recent paper~\cite{DX18}, which proves a quenched localization result for the random walk conditioned to avoid $\Oi$. 
Let us explain the outline of the proof before delving into the details.
Recall that we denote by $\lambda_U$ the smallest Dirichlet eigenvalue of the discrete Laplacian $-\frac{1}{2d}\Delta$ in $U$, which is different from the notation in~\cite{DX18}.


For $\iota,\rho>0$, we introduce a set $\Ei(\iota,\rho)\subset\ball{0}{2N}$ in Definition~\ref{empty-set} which is a collection of large (but $o(\rad)$) boxes where the density of obstacles is low. We will show that $\Ei(\iota,\rho)$ is close to a ball with radius $\rad$ in symmetric difference. Then it follows that the ball is almost free of obstacles.
The purpose of this coarse graining is two-fold: it identifies the regions that actually contribute to the vacant ball (note that the number of open sites in $B(0; 2N)$ is of order $N^d \gg \rad^d$); and it reduces the entropy of the set of obstacle configurations and allows estimates of the form $\P(|\Ei(\iota,\rho)| = V)=(1-p)^{V(1+o(1))}$, as if there is only one configuration of $\Ei(\iota, \rho)$ with given volume $V$.

If we could show that $\Ei(\iota, \rho)$ correctly identifies where the random walk is localized in the sense that, the Dirichlet eigenvalue does not change much if we restrict the walk to $\Ei(\iota, \rho)$, namely,
\begin{equation}
 \lambda_{\ball{0}{2N}\setminus\Oi}\sim \lambda_{\Ei(\iota,\rho)},
\label{eq:EVE}
\end{equation}
then we could (formally) write
\begin{equation}
\begin{split}
 \P\otimes\bP\left(\tau_\Oi>\tau_x^N\right)
&\lesssim \E\left[\exp\left\{-N\lambda_{\ball{0}{2N}\setminus\Oi}\right\}\right]\\
&\lesssim \sum_{V=1}^{(4N+1)^d} \E\left[\exp\left\{-N\lambda_{\Ei(\iota,\rho)}\right\}\colon |\Ei(\iota,\rho)|=V\right]\\
&\approx \sup_{{U}\subset\Z^d} \exp\left\{-N\lambda_{U}-|U|\log\tfrac{1}{1-p}\right\}.
\end{split}
\label{eq:Laplace}
\end{equation}
The last approximation is justified by the fact that $V$ can only take $O(N^d)$ many values, which is of lower order than the exponential asymptotics. We can then apply a quantitative Faber-Krahn inequality to show that the dominant contribution comes from configurations of $\Ei$ that are close to a ball with radius $\rad$.

Unfortunately, it is not easy to prove~\eqref{eq:EVE} directly. Instead, we make a detour by comparing $\Ei(\iota,\rho)$ with a low level set $\Omega_\eta$ of the principal eigenfunction in $\ball{0}{2N}\setminus\Oi$ (see Definition~\ref{def-eign-f}). It is relatively easy to prove that $\lambda_{\Omega_{\eta}}$ well approximates $\lambda_{\ball{0}{2N}\setminus\Oi}$ when $\eta$ is small, and it is also relatively easy to prove that the eigenfunction is small on $\Ei(\iota,\rho)$, and as a consequence $\Ei(\iota,\rho)$ almost contains $\Omega_{\eta}$. See Lemma~\ref{omegaep} for the precise formulation. That lemma essentially allows us to carry out the argument around~\eqref{eq:Laplace} with $\Ei(\iota,\rho)$ replaced by $\Omega_{\eta}$.\\
\medskip

Now let us turn to the formal proof.
\begin{definition}
[Definition 5.1 in~\cite{DX18}]
\label{empty-set}
For $\iota, \rho>0$, a box of the form
\begin{equation}
 K_{\lfloor \iota\rad\rfloor}(x)= x+ [-\lfloor \iota\rad\rfloor,\lfloor \iota\rad\rfloor)^d \text{ for } x \in  (2\lfloor \iota\rad\rfloor+1)\Z^d
\label{eq:def_box}
\end{equation}
is said to be $(\iota\rad,\rho)$-empty if
\begin{equation}
|\Oi \cap K_{\lfloor \iota\rad\rfloor}(x)| \leq \rho |K_{\lfloor \iota\rad \rfloor}(x)|.
\end{equation}
Let $\Ei(\iota,\rho)$ denote the intersection between $\ball{0}{2N}$ and the union of $(\iota\rad,\rho)$-empty boxes.
\end{definition}
We will choose $\rho>0$ small so that $(\iota\rad,\rho)$-empty boxes are rare. As a consequence, we have a rather good control on the volume of $\Ei(\iota,\rho)$.
\begin{lemma}
\label{EnvironmentCost}
For any $\iota,\rho\in(0,1)$ and $V>0$,
\begin{equation}
\P(|\Ei(\iota,\rho)| = V) \le \exp\left\{-V\left(\log {\tfrac{1}{p}}  + 2\rho \log \rho - \frac{\log (3N) }{\lfloor2\iota \rad\rfloor^d}\right)\right\}.
\label{eq:EnvironmentCost}
\end{equation}
\end{lemma}
\begin{proof}
This can be proved in the same way as~\cite[Lemma 5.2]{DX18}.
\end{proof}

\begin{definition}
[Definition 5.3 in \cite{DX18}]
\label{def-eign-f}
Let $f$ be the eigenfunction corresponding to $\lambda_\fregion$ such that $\|f\|_1 = 1$. We extend $f$ to $\Z^d$ by letting $f(v) = 0 $ for $v \in \ball{0}{2N}^c \cup\Oi$ and define
\begin{equation*}
\Omega_\eta = \{ v \in \Z^d \colon f(v) \geq \eta |\Ei(\iota,\rho)|^{-1}\},
\end{equation*}
where we will fix the parameters
\begin{equation}
\label{eq:parameters}
\eta=2\delta_{N,x},\quad\rho=\eta^2,\quad\iota=\eta^{5/2}.
\end{equation}
\end{definition}
Our $\eta$ plays the role of $\epsilon$ in~\cite{DX18}. Noting that $\delta_{N,x}\in [\rad^{-1/5},\epsilon]$ by~\eqref{eq:def_delta} and the assumptions of Lemmas~\ref{lem:almost} and~\ref{lem:slow_cross}, we have
\begin{equation}
\left| 2\rho \log \rho - \frac{\log (3N) }{\lfloor2\iota \rad\rfloor^d}\right|
\le c\eta
\label{eq:vol_ErrorTerms}
\end{equation}
in~\eqref{eq:EnvironmentCost} for all sufficiently large $N$ when $\epsilon$ is small.

As mentioned before, we are going to prove that $\Omega_\eta$ largely coincides with $\Ei(\iota,\rho)$. We start with an \emph{a priori} bound on the eigenvalue $\lambda_{\ball{0}{2N} \setminus \Oi}$ under $\mu_{N,x}$ with $|x|\le \rad^d$, which is a consequence of Proposition~\ref{prop:LDPlower}.
\begin{corollary}
\label{cor:ev_crude}
There exists constant $c_{\ref{cor:ev_crude}}>0$ such that for all $|x| \leq \rad^d$,
\begin{equation}
 \mu_{N,x}(\lambda_{\ball{0}{2N} \setminus \Oi} \geq c_{\ref{cor:ev_crude}} \rad^{-2})
 \leq \exp\left\{- c_{\ref{cor:ev_crude}}^{-1}N^{\frac{d}{d+2}}\right\}.
\end{equation}
\end{corollary}
\begin{proof}
It follows from~\eqref{eq:SG} that
\begin{equation}
\begin{split}
&\P\otimes\bP\left(\lambda_{\ball{0}{2N} \setminus \Oi} \geq C \rad^{-2}, \tau_\Oi>\tau_x^N\right)\\
&\quad\le \E\left[c N^{d/2}(1-\lambda_{\ball{0}{2N} \setminus \Oi})^N\colon \lambda_{\ball{0}{2N} \setminus \Oi} \geq C \rad^{-2}\right]\\
&\quad\le cN^{d/2}\exp\left\{-C\varrho_1^{-2}N^{\frac{d}{d+2}}\right\}.
\end{split}
\end{equation}
Comparing this with~\eqref{eq:pf} and choosing $C>0$ sufficiently large, we obtain the desired result.
\end{proof}

In what follows, we often assume the condition
\begin{equation}
\label{eq:ev_cond}
\lambda_\fregion \leq c_{\text{\upshape{\ref{cor:ev_crude}}}} \rad^{-2}
\end{equation}
appearing in Corollary~\ref{cor:ev_crude}.
\begin{lemma}[Lemmas 5.6 and 5.7 in \cite{DX18}]
\label{omegaep}
Assume~\eqref{eq:ev_cond}. Then there exists $c_{\text{\upshape{\ref{omegaep}}}}>0$ such that 
\begin{align}
|\Omega_{{\eta^{2}}} \setminus\Ei(\iota,\rho)|
&\leq c_{\text{\upshape{\ref{omegaep}}}} {\eta}|\Ei(\iota,\rho)| \label{eq:omegaSizeeee},\\
\lambda_{\Omega_\eta}
&\leq \lambda_{\fregion}(1+c_{\text{\upshape{\ref{omegaep}}}}\eta)\label{eq:omegaEV}.
\end{align}
\end{lemma}
\begin{proof}
This can be proved in the same way as~\cite[Lemmas~5.6 and~5.7]{DX18}.
\end{proof}

Lemmas~\ref{EnvironmentCost} and~\ref{omegaep} provide controls on the volume and eigenvalue of $\Omega_\eta$ since $\Omega_\eta\subset\Omega_{\eta^{2}}$. The reason for considering $\Omega_{\eta^{2}}$ will be explained shortly. We will show that $\Omega_\eta$ is approximately a ball of radius $\rad$ by applying the quantitative Faber--Krahn inequality for the continuum Laplacian eigenvalue in~\cite{BDV15}. In order to apply it in our discrete setting, we need to approximate $\Omega_\eta\subset\mathbb Z^d$ by a continuous set in $\mathbb R^d$ with the volume and eigenvalue controlled. Recall our convention of using boldface letters to denote a subset of $\mathbb R^d$ as well as the smallest Dirichlet eigenvalue of the continuum Laplacian $-\frac{1}{2d}\Delta$.
For the eigenvalue approximation, we use a classical result about the comparison between discrete and continuum eigenvalues in~\cite{Kuttler70}. To this end, we need to introduce slightly enlarged sets
\begin{align}
\Omega_\eta^+ &= \left\{ v \in \Z^d: \min_{x \in \Omega_\eta}|x-v|_\infty < 2\right\},\\
\mathbf{\Omega}_\eta^+ &= \bigcup_{v\in\Omega_\eta^+}\left(v+[-\tfrac12,\tfrac12]^d\right).
\end{align}
Then~\cite[(38)]{Kuttler70} asserts
\begin{equation}
\boldsymbol{\lambda}_{\mathbf{\Omega}_\eta^+}
\leq \lambda_{{\Omega}_\eta} + c \lambda_{{\Omega}_\eta}^2.
\label{eq:con-ev}
\end{equation}

The passage from $\Omega_\eta$ to $\Omega_\eta^+$ is potentially problematic since it can increase the volume substantially when $\Omega_\eta$ has many tiny holes. The following lemma is to solve this problem by showing that $\Omega_\eta^+$ is not much larger than a slightly lower level set $\Omega_{\eta^2}$, for which the volume bound~\eqref{eq:omegaSizeeee} holds.
\begin{lemma}[Lemma 5.8 in \cite{DX18}]
\label{tOmegaSize}
Assume~\eqref{eq:ev_cond}. Then there exists $c_{\text{\upshape{\ref{tOmegaSize}}}}>0$ such that
\begin{equation}
|\Omega_{{\eta}}^+ \setminus  \Omega_{{\eta^{2}}}| \leq  c_{\text{\upshape{\ref{tOmegaSize}}}} \eta |{\Ei(\iota,\rho)}|.
\end{equation}
\end{lemma}
\begin{proof}
This can be proved in the same way as \cite[Lemma~5.8]{DX18}.
\end{proof}

Now we are ready to prove Lemmas~\ref{lem:almost} and~\ref{lem:slow_cross}.
\begin{proof}[Proof of Lemma~\ref{lem:almost}]
In view of Corollary~\ref{cor:ev_crude}, we may assume~\eqref{eq:ev_cond}. The proof is divided into three steps. The last two steps are similar to the proof of~\cite[Lemma~5.9]{DX18} and hence we omit some technical details.
\medskip

\noindent\textbf{Step 1:}
We first prove that $|\boldsymbol{\Omega}_{\eta}^+|$ is not much larger than $|\boldsymbol{B}({0};{\rad})|$, the volume of the Euclidean ball $\boldsymbol{B}({0};{\rad})$, under $\mu_{N,x}$ with high probability.
To this end, we use~\eqref{eq:SG} to obtain
\begin{equation}
\begin{split}
&\P \otimes\bP\left(\tau_\Oi > \tau_x^N, |\Ei(\iota,\rho)| = V\right)\\
&\quad \leq \E\left[C N^{d/2} \exp\left\{-N\lambda_{\fregion}\right\} \colon |\Ei(\iota,\rho)| = V\right].
\end{split}
\label{eq:Vfixed1}
\end{equation}
By~\eqref{eq:ev_cond},~\eqref{eq:omegaEV},~\eqref{eq:con-ev} and the classical Faber--Krahn inequality, when $N$ is sufficiently large, we have
\begin{equation}
\begin{split}
 \lambda_{\fregion}
 &\ge \boldsymbol{\lambda}_{\boldsymbol{\Omega}_\eta^+}(1 - c\eta)\\
 &\ge |\boldsymbol{\Omega}_\eta^+|^{-2/d}\boldsymbol{\lambda}_{\boldsymbol{B}}(1 - c\eta),
\end{split}
\label{eq:FK}
\end{equation}
where $\boldsymbol{B}$ is the ball with unit volume in $\R^d$ centered at the origin. Moreover, by Lemmas~\ref{omegaep} and~\ref{tOmegaSize} and the fact that $|\Omega_\eta^+|=|\boldsymbol{\Omega}_\eta^+|$, it follows that on $\{|\Ei(\iota,\rho)| = V\}$,
\begin{equation}
\begin{split}
|\boldsymbol{\Omega}_\eta^+|
&\le |\Ei(\iota,\rho)|+|{\Omega}_{\eta^2}\setminus \Ei(\iota,\rho)|+|{\Omega}_\eta^+\setminus {\Omega}_{\eta^2}|\\
&\le |V|(1+c\eta).
\end{split}
\label{eq:bOmegaVol}
\end{equation}
Substituting~\eqref{eq:FK} and~\eqref{eq:bOmegaVol} into~\eqref{eq:Vfixed1} and using  Lemma~\ref{EnvironmentCost} and~\eqref{eq:vol_ErrorTerms}, we find that
\begin{equation}
\begin{split}
&\P \otimes\bP\left(\tau_\Oi > \tau_x^N, |\Ei(\iota,\rho)| = V\right)\\
&\quad \le C N^{d/2} \exp\left\{-NV^{-2/d}\boldsymbol{\lambda}_{\boldsymbol{B}}(1 - c\eta)\right\} \P\left( |\Ei(\iota,\rho)| = V\right)\\
&\quad \le C N^{d/2} \exp\left\{-NV^{-2/d}\boldsymbol{\lambda}_{\boldsymbol{B}}(1 - c\eta) + V(\log p  + \eta)\right\}\\
&\quad\le \exp\left\{-\left(c(d,p)-C\eta+c\left(\frac{V}{|\boldsymbol{B}({0};{\rad})|}-1\right)^2\right)N^{\frac{d}{d+2}} \right\},
\end{split}
\label{eq:Vfixed2}
\end{equation}
where we used a second order Taylor expansion for the function $V\mapsto NV^{-2/d}\boldsymbol{\lambda}_{\boldsymbol{B}(0;r)}+V\log {\frac{1}{p}}$ at $V=|\boldsymbol{B}({0};{\rad})|$, where it takes its minimal value $c(d,p)N^{\frac{d}{d+2}}$  (see the discussion following~\eqref{eq:DV79}).

Now if we suppose $|\boldsymbol{\Omega}^+_\eta|\ge |\boldsymbol{B}({0};{\rad})|+\eta^{1/3}\rad^d$ and $|\Ei(\iota,\rho)| = V$, then by~\eqref{eq:bOmegaVol} we have $V\ge |\boldsymbol{B}({0};{\rad})|+\tfrac12\eta^{1/3}\rad^d$ and hence
\begin{equation}
\left(\frac{V}{|\boldsymbol{B}({0};{\rad})|}-1\right)^2\ge c\eta^{2/3}.
\label{eq:non_optimal}
\end{equation}
Since the number of possible values of $V$ is bounded by $(4N+1)^d$ and $\eta=2\delta_{N,x}$, comparing~\eqref{eq:Vfixed2} with~\eqref{eq:pf} shows that
\begin{equation}
\begin{split}
& \mu_{N,x}\left(|\boldsymbol{\Omega}_{\eta}^+|\ge |\boldsymbol{B}({0};{\rad})|+ \eta^{1/3}\rad^d\right)\\
&\quad \le \sum_{V:\text{\eqref{eq:non_optimal}}}
\frac{\P\otimes\bP(\tau_\Oi> \tau_x^N, |\Ei(\iota,\rho)| = V)}{\P\otimes\bP(\tau_\Oi>\tau_x^N)}\\
&\quad \le (4N+1)^d \exp\left\{-c(2\delta_{N,x})^{2/3} N^{\tfrac{d}{d+2}}+{\rm dist}_\beta\left(x, \ball{0}{2\rad}\right)+\epsilon \delta_{N,x}N^{\frac{d}{d+2}}\right\}\\
&\quad \le  \exp\left\{-\delta_{N,x} N^{\tfrac{d}{d+2}}\right\}
\end{split}
\label{eq:VOLconst1}
\end{equation}
for all sufficiently large $N$ when $\epsilon>0$ is small. This gives the bound we need on $|\boldsymbol{\Omega}_\eta^+|$.
\medskip

\noindent\textbf{Step 2:}
Next, we prove that with high probability under $\mu_{N,x}$, $\boldsymbol{\lambda}_{\boldsymbol{\Omega}_\eta^+}$ is not much larger than ${\boldsymbol{\lambda}}_{\boldsymbol{B}({0};{\rad})}$. As a consequence, we will also see that $|\boldsymbol{\Omega}_\eta^+|$ is not much smaller than $|\boldsymbol{B}({0};{\rad})|$. If we replace the condition $|\Ei(\iota,\rho)|=V$ in~\eqref{eq:Vfixed1} by $\boldsymbol{\lambda}_{\boldsymbol{\Omega}_\eta^+}\ge {\boldsymbol{\lambda}}_{\boldsymbol{B}({0};{\rad})}(1 + \eta^{1/2})$, then by using the first line of~\eqref{eq:FK}, we get
\begin{equation}
\P\otimes\bP\left(\tau_{\Oi}>\tau_x^N,\boldsymbol{\lambda}_{\boldsymbol{\Omega}_\eta^+}\ge {\boldsymbol{\lambda}}_{\boldsymbol{B}({0};{\rad})}(1 + \eta^{1/2})\right)
\le \exp\left\{-\left(c(d,p)+\tfrac{1}{2}\eta^{1/2}\right)N^{\frac{d}{d+2}} \right\}.
\end{equation}
Recalling $\eta=2\delta_{N,x}$ and comparing the above with~\eqref{eq:pf} again, we conclude that
\begin{equation}
\mu_{N,x}\left(\boldsymbol{\lambda}_{\boldsymbol{\Omega}_\eta^+}\ge {\boldsymbol{\lambda}}_{\boldsymbol{B}({0};{\rad})}(1 + \eta^{1/2})\right)
\le \exp\left\{-\delta_{N,x}^{1/2}N^{\frac{d}{d+2}} \right\}
\label{eq:EVconst1}
\end{equation}
for all sufficiently large $N$ when $\epsilon>0$ is small. On the complementary event $\{\boldsymbol{\lambda}_{\boldsymbol{\Omega}_\eta^+}< {\boldsymbol{\lambda}}_{\boldsymbol{B}({0};{\rad})}(1 + \eta^{1/2})\}$, the classical Faber--Krahn inequality implies
\begin{equation}
 |\mathbf{\Omega}_{\eta}^+|\ge |\boldsymbol{B}({0};{\rad})|(1-c\eta^{1/2}),
\label{eq:VOLconst2}
\end{equation}
which complements the upper bound in~\eqref{eq:VOLconst1}.
\medskip

\noindent\textbf{Step 3:}
We prove that $\mathbf{\Omega}_{\eta}^+$ is well approximated by a ball of radius almost $\rad$ by using the quantitative Faber--Krahn inequality in~\cite{BDV15}. Let us recall that ~\cite[MAIN THEOREM]{BDV15} asserts
\begin{equation}
\inf\left\{\frac{|\boldsymbol{B'}\triangle\boldsymbol{\Omega}|}{|\boldsymbol{B'}|}
\colon
\boldsymbol{B'}\text{ is any ball with }|\boldsymbol{B'}|=|\boldsymbol{\Omega}|\right\}^{2}
\le c \left(|\boldsymbol{\Omega}|^{2/d}\boldsymbol{\lambda}_{\boldsymbol{\Omega}}
-\boldsymbol{\lambda}_{\boldsymbol{B}}\right),
\label{eq:BDV15}
\end{equation}
where $\boldsymbol{B}$ is the ball with unit volume in $\R^d$ centered at the origin (see~\eqref{eq:FK}). By~\eqref{eq:VOLconst1},~\eqref{eq:VOLconst2} and~\eqref{eq:EVconst1} in the previous steps, we may assume
\begin{align}
&\left||\boldsymbol{\Omega}_{\eta}^+|-|\boldsymbol{B}({0};{\rad})|\right|
\le \eta^{1/3}\rad^d,\label{eq:VOLconst}\\
&\boldsymbol{\lambda}_{\boldsymbol{\Omega}_\eta^+} \le {\boldsymbol{\lambda}}_{\boldsymbol{B}({0};{\rad})}(1 + \eta^{1/2}).\label{eq:EVconst}
\end{align}
In particular, it follows that
\begin{equation}
\begin{split}
|\boldsymbol{\Omega}_\eta^+|^{2/d} \boldsymbol{\lambda}_{\boldsymbol{\Omega}_\eta^+} - {\boldsymbol{\lambda}}_{\boldsymbol{B}}
&=|\boldsymbol{\Omega}_\eta^+|^{2/d} \boldsymbol{\lambda}_{\boldsymbol{\Omega}_\eta^+} - |\boldsymbol{B}({0};{\rad})|^{2/d}{\boldsymbol{\lambda}}_{\boldsymbol{B}({0};{\rad})}\\
&\le c\eta^{1/3}.
\end{split}
\end{equation}
Substituting this into~\eqref{eq:BDV15}, we can find a ball $\boldsymbol{B}_\eta$ such that $|\boldsymbol{B}_\eta|=|\mathbf{\Omega}_\eta^+|$ and $|\mathbf{\Omega}_\eta^+ \triangle \boldsymbol{B}_\eta| \le c \eta^{1/6}\rad^d$. Setting $\cent$ as the center of $\boldsymbol{B}_\eta$ and using~\eqref{eq:VOLconst} again, we find
\begin{equation}
|\boldsymbol{B}(\cent;\rad) \triangle \boldsymbol{B}_\eta|
\le c \eta^{1/6}\rad^d.
\label{eq:2balls}
\end{equation}
\smallskip

\noindent\textbf{Step 4:}
Finally, we prove that $B_\eta=\boldsymbol{B}_\eta\cap\Z^d$ is almost free of obstacles. To this end, we first show that
\begin{equation}
|B_\eta\setminus \Ei(\iota,\rho)|
\le |B_\eta\setminus \Omega_{\eta}^+|
+|\Omega_{\eta}^+\setminus \Ei(\iota,\rho)|
\label{eq:OmegaMid}
\end{equation}
is small.
Recall that~\eqref{eq:VOLconst1} shows $|\Ei(\iota,\rho)| \le c\rad^d$ with $\mu_{N,x}$ probability greater than $1-\exp\{-\delta_{N,x}N^{\frac{d}{d+2}}\}$. Under this condition, the second term on the right-hand side is smaller than $c\eta\rad^d$ due to Lemmas~\ref{omegaep} and~\ref{tOmegaSize}. For the first term, note first that
\begin{equation}
 |B_\eta\setminus \Omega_{\eta}^+| = |B_\eta|-|\Omega_{\eta}^+|+|\Omega_{\eta}^+\setminus B_\eta|.
\label{eq:interchange}
\end{equation}
Since we know $|\boldsymbol{B}_\eta|=|\mathbf{\Omega}_\eta^+|$, $|\Omega_\eta^+|=|\mathbf{\Omega}_\eta^+|$ and $|B_\eta|\le |\boldsymbol{B}_\eta|+c\rad^{d-1}$, we only need to prove that $|\Omega_{\eta}^+\setminus B_\eta|$ is small. Since $\boldsymbol\Omega_{\eta}^+$ is a union of cubes $\{y+{[-\frac12,\frac12]^d}\colon y\in\Omega_\eta^+\}$ and $|(y+{[-\frac12,\frac12]^d})\setminus \boldsymbol{B}_\eta|\ge \tfrac12$ for any $y\in \Omega_{\eta}^+\setminus B_\eta$, it follows that
\begin{equation}
\begin{split}
|\Omega_{\eta}^+\setminus B_\eta|
&= \sum_{y\in \Omega_{\eta}^+ \setminus B_\eta} |y+[-\tfrac12,\tfrac12]^d|\\
&\le 2\sum_{y\in \Omega_{\eta}^+} |(y+[-\tfrac12,\tfrac12]^d)\setminus \boldsymbol{B}_\eta|\\
&\le 2|\boldsymbol{\Omega}_\eta^+\setminus \boldsymbol{B}_\eta|.
\end{split}
\end{equation}
This last line is shown to be bounded by $c \eta^{1/6}\rad^d$ in Step~3. Therefore we conclude that~\eqref{eq:OmegaMid} is bounded by $c \eta^{1/6}\rad^d$. Recalling~\eqref{eq:2balls} and Definition~\ref{empty-set}, one can easily check that this implies~\eqref{eq:almost}.
\end{proof}
In fact,~\eqref{eq:VOLconst1} in Step 1 shows that $|\Ei(\iota,\rho)|$ is close to $|\ball{0}{\rad}|$ and hence by arguing as in~\eqref{eq:interchange}, we find the following:
\begin{corollary}
\label{cor:symmdiff}
Under the same assumption as in Lemma~\ref{lem:almost}, there exists $c_{\text{\upshape{\ref{cor:symmdiff}}}}>0$ such that for $\iota,\rho$ as in~\eqref{eq:parameters},
\begin{equation}
\mu_{N,x}\left(|\Ei(\iota,\rho)\triangle \ball{\cent}{\rad}|
 \ge \delta_{N,x}^{c_{\text{\upshape{\ref{cor:symmdiff}}}}}N^{\frac{d}{d+2}}\right)
\le \exp\left\{-{c_{\text{\upshape{\ref{cor:symmdiff}}}}} \delta_{N,x} N^{\frac{d}{d+2}}\right\}
\label{eq:symmdiff}
\end{equation}
for all sufficiently large $N$.
\end{corollary}

\begin{proof}
[Proof of Lemma~\ref{lem:slow_cross}]
We only sketch the argument since the proof is almost identical to~\cite[Lemma~6.1]{DX18}. Thanks to Corollary~\ref{cor:symmdiff}, it suffices to prove that~\eqref{eq:slow_cross} holds on the event
\begin{equation}
\label{eq:typical_env}
\left\{|\Ei(\iota,\rho)\triangle \ball{\cent}{\rad}|
 <\delta_{N,x}^{c_{\text{\upshape{\ref{cor:symmdiff}}}}}N^{\frac{d}{d+2}}
\right\}.
\end{equation}
On this event, we have
\begin{equation}
\left|\Ei(\iota,\rho)\setminus B^-\right|
\le \delta_{N,x}^C \rad^d
\label{eq:outside_thin}
\end{equation}
for some $C \in (0,1)$, where $B^-$ is defined in~\eqref{eq:def_B-}.

Let us first assume that $S_{[0,t]}$ visits $(\Ei(\iota,\rho)\setminus B^-)^c$ more than $t/4$ times, which is natural since $|\Ei(\iota,\rho)\setminus B^-|$ has a small volume by~\eqref{eq:outside_thin}. Then we can extract at least $\tfrac14 \delta_{N,x}^{-5}\rad^{-2}t$ different times $k_i$'s in $[0,t]$ which satisfy
\begin{itemize}
 \item $S_{k_i}\not\in \Ei(\iota,\rho)\setminus B^-$ and
 \item $k_{i+1}-k_i \ge \delta_{N,x}^5\rad^2$
\end{itemize}
for all $i\le \tfrac14 \delta_{N,x}^{-5}\rad^{-2}t-1$. Recalling the definition of $\Ei(\iota,\rho)$, one can prove that for each $k_i$, the probability for the random walk to avoid $\Oi\cup B^-$ until next $k_{i+1}$ is smaller than $1-c\rho$ (see~\cite[(5.6)]{DX18}). Therefore we obtain
\begin{equation}
\begin{split}
&\sup_{y\in\ball{0}{2N}}\bP_y\left(S_{[0,t]}\text{ visits $(\Ei(\iota,\rho)\setminus B^-)^c$ more than $t/4$ times and }\tau_{\Oi\cup B^-}>t\right)\\
&\quad \le (1-c\rho)^{\frac14 \delta_{N,x}^{-5}\rad^{-2}t-1}\\
&\quad \le \exp\left\{-c\delta_{N,x}^{-3}\rad^{-2}t\right\}
\end{split}
\label{many_exits}
\end{equation}
by recalling $\rho=(2\delta_{N,x})^2$.

It remains to show that under the condition~\eqref{eq:outside_thin},
\begin{equation}
\begin{split}
\sup_{y\in\ball{0}{2N}}\bP_y\left(S_{[0,t]}\text{ visits $(\Ei(\iota,\rho)\setminus B^-)^c$ less than $t/4$ times}\right)
\le \exp\left\{-\delta_{N,x}^{-c}\rad^{-2} t\right\}
\end{split}
\label{few_exits}
\end{equation}
for some $c>0$. To this end, we divide $[0,t]$ into sub-intervals of size $M^{2}\delta_{N,x}^{2C/d} \rad^{2}$ for some large $M>0$ to be determined later. We call a sub-interval \emph{successful} if the random walk spends more than half of the time in $(\Ei(\iota,\rho)\setminus B^-)^c$. If half of the intervals are successful, then the random walk visits $(\Ei(\iota,\rho)\setminus B^-)^c$ at least $t/4$ times.

For any $u \ge M\delta_{N,x}^{2C/d} \rad^{2}$, a well-known bound on the transition probability $\sup_{y,z\in\Z^d}\bP_z(S_u=y)\le cM^{-d/2}\delta_{N,x}^{-C} \rad^{-d}$ and~\eqref{eq:outside_thin} imply that
\begin{equation}
 \sup_{z\in\Z^d}\bP_z(S_u\in\Ei(\iota,\rho)\setminus B^-)\le cM^{-d/2}.
\end{equation}
From this and the so-called \emph{first moment method} (applied to the number of visits to $\Ei(\iota,\rho)\cup B^-$), one can easily deduce that the probability for an interval $\delta_{N,x}^{2C/d} \rad^{2}[kM^2,(k+1)M^{2}]$ to be successful is more than $1/2$ for large $M$ (see~\cite[Lemma~6.4]{DX18}). Since there are $M^{-2}\delta_{N,x}^{-2C/d} \rad^{-2}t$ intervals that intersect $[0,t]$, a simple large deviation estimate yields
\begin{equation}
 \sup_{z\in\Z^d}\bP_z(\text{half of those intervals are not successful})
\le \exp\left\{-cM^{-2}\delta_{N,x}^{-2C/d} \rad^{-2}t\right\}.
\end{equation}
Combining~\eqref{few_exits} and~\eqref{many_exits}, we get Lemma~\ref{lem:slow_cross}.
\end{proof}

\subsection{Sketch proof of Lemma~\ref{lem:odensity}}
\label{sec:odensity}

This is an analogue of~\cite[Lemma~2.1]{DFSX18} and can be proved by almost the same argument. We recall the outline and indicate where we need an additional argument.

The proof of~\cite[Lemma~2.1]{DFSX18} is based on an environment and path switching argument. Suppose that $v\in\Oi$ and $|\Oi\cap\ball{v}{l}|<\delta|\ball{v}{l}|$. First, if the random walk frequently visits $\ball{v}{l/2}$, then we simply remove all the obstacles in $\ball{v}{l}$. This causes a cost in $\P$-probability but not too much since $\delta$ is small. On the other hand, we gain a lot in $\bP$-probability since $\ball{v}{l/2}$ is visited frequently and it turns out that the gain beats the cost. It follows that
\begin{equation}
E_l^\delta(v)\cap\{\ball{v}{l/2}\text{ is visited frequently}\}
\end{equation}
is much less likely than $\{\Oi\cap \ball{v}{l}=\emptyset\}$ under $\P\otimes\bP$ and hence $\mu_N$.
Second, if the random walk rarely visits $\ball{v}{l}$, then we deform the random walk paths to avoid $\ball{v}{l/2}$. This causes a cost in $\bP$-probability but not too much since the random walk visits $\ball{v}{l/2}$ only rarely. On the other hand, after this operation, we can change the configuration of $\Oi\cap \ball{v}{l/2}$ to a typical one. As we started from an atypical low density configuration, we gain a lot in $\P$-probability and it turns out that the gain beats the cost. It follows that
\begin{equation}
E_l^\delta(v)\cap\{\ball{v}{l/2}\text{ is visited rarely}\}
\end{equation}
is much less likely than $\{|\Oi\cap \ball{v}{l}|\ge \delta|\ball{v}{l}\}$ under $\mu_N$.

When $0\in\ball{v}{l/2}$, the argument in the second case, where the random walk rarely visits $\ball{v}{l}$, requires a modification since we cannot change the starting point of the random walk. In this situation, we create a one-dimensional path from $0$ to $\partial\ball{0}{l/2}$ free of obstacles and force the random walk to follow that path. The only difference in the setting of the present article is that we have the same problem when $x\in B(v,l/2)$, since we cannot change the endpoint. But this can be treated in the same way as the case $0\in\ball{v}{l/2}$.

Finally, in the following remark, we explain a technical point which forces us to work under $\mu_{N,x}(\cdot)=\P\otimes\bP(\cdot\mid \tau_\Oi>\tau_x^N)$ instead of $\P\otimes\bP(\cdot\mid \tau_\Oi>N, S_N=x)$.
\begin{remark}
 \label{rem:switch}
In the case that $\ball{v}{l/2}$ is rarely visited, we deform the random walk path to avoid $\ball{v}{l/2}$, which may lengthen the path. Therefore the condition $S_N=x$ is not preserved by the above argument but $\tau_\Oi>\tau_x^N$ is. This is why we work with $\mu_{N,x}$.
\end{remark}

\section{Time spent outside the vacant ball}
\label{sec:outsideB}
In this section, we prove several results concerning the behavior of the random walk outside the vacant ball $B^-$ defined in~\eqref{eq:def_B-}. The first one, Proposition~\ref{prop:outsideB} to be proved in Section~\ref{sec:time_outside}, shows that the random walk does not spend too much time before the first visit and after the last visit to the ball $B^-$. The second one, Proposition~\ref{prop:confine} to be proved in Section~\ref{sec:confine}, shows that between the first and last visit to $B^-$, the random walk is confined in a slightly larger ball. By the same argument, we show in Corollary~\ref{cor:rho^2} that the random walk returns to $B^-$ frequently between the first and the last visit to $B^-$.

Let us write $\tau^{\leftarrow}_{B^-}$ for the last visit to $B^-$ before $\tau_x^N$, which is the first hitting time of $B^-$ by the time-reversed random walk.

\subsection{First and last visits to the vacant ball}
\label{sec:time_outside}
\begin{proposition}
\label{prop:outsideB}
Let $\delta_{N,x}$ and $\cent$ be as in~\eqref{eq:def_delta} and Proposition~\ref{prop:vacant}, respectively.
There exist $c_{\ref{prop:outsideB}}>0$ such that when $\epsilon>0$ is small depending on $d$ and $p$ and $|x|\le \epsilon\rad^d$,
\begin{align}
\label{eq:hitting}
 \mu_{N,x}\left(\tau_{B^-} \ge \delta_{N,x}^{c_{\ref{prop:outsideB}}}|\cent|_1\rad^2 \right)
\le \exp\left\{-\tfrac12(\log N)^2\right\}
\end{align}
and
\begin{align}
\label{eq:leaving}
 \mu_{N,x}\left(\tau_x^N-\tau^{\leftarrow}_{B^-} \ge \delta_{N,x}^{c_{\ref{prop:outsideB}}}|x-\cent|_1\rad^2 \right)
\le \exp\left\{-\tfrac12(\log N)^2\right\}
\end{align}
for all sufficiently large $N$.
\end{proposition}
\begin{proof}
We give a proof of~\eqref{eq:hitting}. One can prove~\eqref{eq:leaving} similarly by considering the time-reversed random walk. Thanks to Proposition~\ref{prop:vacant} and Lemma~\ref{lem:slow_cross}, we may assume that there exists $z\in\ball{0}{2N}$ such that
\begin{align}
&\Oi\cap\ball{z}{(1-\delta_{N,x}^{c_{\ref{prop:vacant}}})\rad}=\emptyset,\label{eq:vacant2}\\
&\bP_y\left(S_{[0,t]}\cap \left(\Oi \cup B^-(z)\right)=\emptyset\right)
\leq \exp\left\{-\delta_{N,x}^{-c_{\text{\upshape{\ref{lem:slow_cross}}}}} \rad^{-2} t\right\}
\label{eq:forest2}
\end{align}
for all $y\in\ball{0}{2N}$ and $t\ge \delta_{N,x}^{c_{\text{\upshape{\ref{lem:slow_cross}}}}} (\log N)^2\rad^{2}$. In particular, it follows that
\begin{equation}
\bP\left(\tau_\Oi\wedge \tau_{B^-(z)}>N/2\right)
\le \exp\left\{-c\delta_{N,x}^{-c_{\ref{lem:slow_cross}}}N^{\frac{d}{d+2}}\right\}.
\end{equation}
Comparing with Proposition~\ref{prop:LDPlower} and using that $\lim_{\epsilon\to0}\lim_{N\to\infty}\delta_{N,x}=0$, we find that the random walk hits $B^-(z)$ with high probability:
\begin{equation}
\mu_{N,x}\left({\text{\eqref{eq:vacant2}, \eqref{eq:forest2}},}\ \tau_{B^-(z)}>N/2\right)
\le \exp\left\{-c\delta_{N,x}^{-{c_{\ref{lem:slow_cross}}}}N^{\frac{d}{d+2}}\right\}
\end{equation}
for all sufficiently large $N$ when $\epsilon$ is small. Now let us fix $c_{\ref{prop:outsideB}}<c_{\ref{lem:slow_cross}}$, $z\in\ball{0}{2N}$, $y\in B^-(z)$ and $n\in [\delta_{N,x}^{{c_{\ref{prop:outsideB}}}}|z|_1\rad^2,N/2]$. To prove~\eqref{eq:hitting}, it suffices to show that
\begin{equation}
 \mu_{N,x}\left(\text{\eqref{eq:vacant2}, \eqref{eq:forest2}}, \tau_{B^-(z)}=n, S_n=y\right)
\le \exp\left\{-(\log N)^3\right\}
\label{eq:yzn_fixed}
\end{equation}
since the number of possible choices of $(y,z,n)$ is polynomial in $N$. We will only consider $|z|_1\ge \rad/2$ since otherwise $\tau_{B^-(z)}=0$ almost surely under $\mu_{N,x}$.
By using the Markov property at time $n$ and~\eqref{eq:forest2}, we find that
\begin{equation}
\begin{split}
\bP\left(\tau_{B^-(z)}=n, S_n=y, \tau_\Oi>\tau_x^N\right)
& \le \bP\left(\tau_\Oi>\tau_{B^-(z)}=n, S_n=y\right)\bP_y\left(\tau_\Oi>\tau_x^{N-n}\right)\\
& \le \exp\left\{-\delta_{N,x}^{-c_{\ref{lem:slow_cross}}}n\rad^{-2}\right\}\bP_y\left(\tau_\Oi>\tau_x^{N-n}\right).
\end{split}
\label{eq:stayforest}
\end{equation}

In order to compare this with the partition function, let us fix a nearest neighbor path $\pi(0,z)$ of length $|z|_1$ connecting $0$ and $z$ and consider the events
\begin{align}
E_1&=\{S_{[0,|z|_1]}=\pi(0,z)\},\\
E_2&=\left\{S_{[|z|_1, n-\rad^2]}\subset\ball{z}{\rad/2}\right\},\\
E_3&=\left\{S_{[n-\rad^2,n]}\subset B({z};{(1-\delta_{N,x}^{c_{\ref{prop:vacant}}})\rad}), S_n=y\right\}.
\end{align}
Note that on the event $\{\text{\eqref{eq:vacant2}, }\Oi\cap\pi(0,z)=\emptyset\}$, we have $E_1\cap E_2\cap E_3\subset\{S_n=y,\tau_\Oi>n\}$.
Therefore by the Markov property and $|z|_1 \le \delta_{N,x}^{-c_{\ref{prop:outsideB}}}n\rad^{-2}$, we get
\begin{equation}
\begin{split}
&\bP\left(\tau_\Oi>\tau_x^N\right)\\
&\quad\ge \bP(E_1)\bP(E_2\mid S_{|z|_1}=z)\inf_{w\in\ball{z}{\rad/2}}\bP(E_3\mid S_{n-\rad^2}=w)\bP_y\left(\tau_\Oi>\tau_x^{N-n}\right)\\
&\quad\ge\left(\frac{1}{2d}\right)^{|z|_1}\exp\left\{-c(n-\rad^2-|z|_1)\rad^{-2}\right\} \frac{c}{\rad^{d+1}} \bP_y\left(\tau_\Oi>\tau_x^{N-n}\right)\\
&\quad\ge\exp\left\{-c\delta_{N,x}^{-c_{\ref{prop:outsideB}}}n\rad^{-2}\right\} \bP_y\left(\tau_\Oi>\tau_x^{N-n}\right),
\end{split}
\label{eq:switched}
\end{equation}
where we have used the random walk estimate~\eqref{eq:ballHK} for the second and third factors in the second line. 

Now we use a slight variant of the switching argument in~\cite{DFSX18}. We first use the Markov property at time $n$ and~\eqref{eq:stayforest} to obtain
\begin{equation}
\begin{split}
&\P\otimes\bP\left(\text{\eqref{eq:vacant2}, \eqref{eq:forest2}}, \tau_{B^-(z)}=n, S_n=y, \tau_\Oi>\tau_x^N\right)\\
&\quad\le \exp\left\{-c\delta_{N,x}^{-c_{\ref{lem:slow_cross}}}n\rad^{-2}\right\}\E\left[\bP_y\left(\tau_\Oi>\tau_x^{N-n}\right)\colon \text{\eqref{eq:vacant2}}\right].
\end{split}
\label{eq:switch1}
\end{equation}
Then we ``switch'' a given $\Oi$ satisfying~\eqref{eq:vacant2} by removing the obstacles on $\pi(0,z)$. Since $\bP_y\left(\tau_\Oi>\tau_x^{N-n}\right)$, $\text{\eqref{eq:vacant2}}$ and $\{\pi(0,z)\cap\Oi=\emptyset\}$ are all decreasing in $\Oi$, we can use the FKG inequality to obtain
\begin{equation}
\begin{split}
& \E\left[\bP_y\left(\tau_\Oi>\tau_x^{N-n}\right)\colon \text{\eqref{eq:vacant2}}\right]\\
&\quad \le p^{-|z|_1}\E\left[\bP_y\left(\tau_\Oi>\tau_x^{N-n}\right)\colon \text{\eqref{eq:vacant2}, }\Oi\cap\pi(0,z)=\emptyset\right].
\end{split}
\end{equation}
Substituting this and~\eqref{eq:switched} into~\eqref{eq:switch1} and recalling $c_{\ref{prop:outsideB}}<c_{\ref{lem:slow_cross}}$ and $|z|_1 \le \delta_{N,x}^{-c_{\ref{prop:outsideB}}}n\rad^{-2}$ again, we find
\begin{equation}
\begin{split}
&\P\otimes\bP\left(\text{\eqref{eq:vacant2}, \eqref{eq:forest2}}, \tau_{B^-(z)}=n, S_n=y, \tau_\Oi>\tau_x^N\right)\\
&\quad\le \exp\left\{-c\delta_{N,x}^{-c_{\ref{lem:slow_cross}}}n\rad^{-2}\right\}
\E\left[\bP_y\left(\tau_\Oi>\tau_x^{N-n}\right)\colon \text{\eqref{eq:vacant2}, }\Oi\cap\pi(0,z)=\emptyset\right].\\
&\quad \le \exp\left\{-c'\delta_{N,x}^{-c_{\ref{lem:slow_cross}}}n\rad^{-2}\right\}\P\otimes\bP\left(\tau_\Oi>\tau_x^N\right)
\end{split}
\end{equation}
for all sufficiently large $N$. Recalling $n\ge \delta_{N,x}^{{c_{\ref{prop:outsideB}}}}|z|_1\rad^2$ and $|z|_1\ge \rad/2$, this implies~\eqref{eq:yzn_fixed} and we are done.
\end{proof}
\subsection{Confinement between the first and last visits to the vacant ball}
\label{sec:confine}
Let us introduce a ball concentric to $B^-$ with a larger radius by
\begin{equation}
\begin{split}
B^+(z)&=\ball{z}{(1+\delta_{N,x}^{c_{\ref{lem:slow_cross}}{/2}}{(\log N)^3})\rad},\\
B^+&=B^+(\cent).
\end{split}
\label{eq:def_B+}
\end{equation}
Note that by our definition of $\delta_{N,x}$ in~\eqref{eq:def_delta}, this is much larger than $B^-$, see~\eqref{eq:def_B-}, when $|x|$ is close to $\rad^d$. We will explain the reason in Remark~\ref{rem:LargeBall}.
In the following proposition, we show that $S_{[\tau_{B^-}, \tau^{\leftarrow}_{B^-}]}$ is confined in $B^+$ with high probability under $\mu_{N,x}$.
\begin{proposition}
\label{prop:confine}
When $\epsilon>0$ is small depending on $d$ and $p$ and $|x|\le \epsilon\rad^d$,
\begin{equation}
\mu_{N,x}\left(S_{[\tau_{B^-},\tau^{\leftarrow}_{B^-}]} \not\subset B^+\right)\\
\le \exp\left\{-\tfrac13(\log N)^2\right\}
\label{eq:conf_middle}
\end{equation}
for all sufficiently large $N$.
\end{proposition}
\begin{proof}
Throughout this proof, we assume that~\eqref{eq:vacant} and~\eqref{eq:slow_cross} hold, that is, there exists a vacant ball of radius almost $\rad$ and the outside is dangerous for the random walk.

Suppose that $S_{[\tau_{B^-},\tau^{\leftarrow}_{B^-}]}\not\subset B^+$. Then since we know $0\le \tau_{B^-}<\tau^{\leftarrow}_B<\tau_\Oi^N$ from Proposition~\ref{prop:outsideB}, there exist $[t_1,t_2]\subset [\tau_{B^-}, \tau^{\leftarrow}_{B^-}]$ such that $ S_{t_1}, S_{t_2} \in \partial B^-$,
\begin{equation}
\label{eq:crossing}
S_{[t_1,t_2]}\cap\left(\Oi\cup{B^-}\right)=\emptyset\text{ and }
S_{[t_1,t_2]}\cap (B^+)^c \neq\emptyset.
\end{equation}
Therefore, by using the union bound and the Markov property, we have
\begin{equation}
\begin{split}
& \bP\left(S_{[\tau_{B^-},\tau^{\leftarrow}_{B^-}]}\not\subset B^+, \tau_\Oi>\tau^N_x\right)\\
&\quad \le \sum_{t_1,t_2,x_1,x_2}
\bP\left(S_{t_1}=x_1,\tau_\Oi>t_1\right)
\bP_{x_1}(\eqref{eq:crossing}, S_{t_2-t_1}=x_2)
\bP_{x_2}\left(\tau_\Oi>\tau^{N-t_2}_x\right),
\end{split}
\label{eq:decomposition}
\end{equation}
where the above sum runs over $0\le t_1<t_2\le 2N$ and $x_1,x_2\in \partial B^-$. 
We are going to show that the middle term on the right-hand side of~\eqref{eq:decomposition} is much smaller than the probability of
\begin{equation}
S_{[t_1,t_2+\rad^2]}\subset \ball{\cent}{(1-\delta_{N,x}^{c_{\ref{prop:vacant}}})\rad}\text{ and } S_{t_2+\rad^2}=x_2.
\label{eq:stay_inside2}
\end{equation}
Since we assumed~\eqref{eq:vacant}, this in particular implies $S_{[t_1,t_2+\rad^2]}\cap\Oi=\emptyset$. Let us first get a lower bound on the probability of~\eqref{eq:stay_inside2}. Using~\eqref{eq:ballHK}, we obtain
\begin{equation}
 \min_{x_1, x_2\in \partial B^-}
 \bP_{x_1}\left(\eqref{eq:stay_inside2}\right)
\ge \frac{c}{\rad^{d+2}
}\exp\left\{-c^{-1}(t_2-t_1)\rad^{-2}\right\}.
\label{eq:stay_inside}
\end{equation}

Next we get an upper bound on the probability of~\eqref{eq:crossing} which splits into two cases.
\medskip

\noindent
\textbf{Case 1}: $t_2-t_1\ge \delta_{N,x}^{c_{\ref{lem:slow_cross}}}{(\log N)^2}\rad^2$. In this case, we only consider the first condition in~\eqref{eq:crossing}. Then since we are assuming~\eqref{eq:slow_cross}, we have

\begin{equation}
\begin{split}
 \bP_{x_1}(\eqref{eq:crossing})
& \le \exp\left\{-\delta_{N,x}^{-c_{\ref{lem:slow_cross}}}\rad^{-2}(t_2-t_1)\right\}\\
& \le {\exp\left\{-(\log N)^2\right\}\bP_{x_1}\left(\eqref{eq:stay_inside2}\right)}
\end{split}
\label{eq:slow}
\end{equation}
for all sufficiently large $N$ when $\epsilon$ is small.\\
\smallskip

\noindent\textbf{Case 2}: $t_2-t_1\le {\delta_{N,x}^{c_{\ref{lem:slow_cross}}}(\log N)^2}\rad^2$. In this case, we only consider the second condition in~\eqref{eq:crossing}, which implies that the maximal displacement of the random walk on $[t_1,t_2]$ is larger than $\delta_{N,x}^{c_{\ref{lem:slow_cross}}{/2}}{(\log N)^3}\rad$. Then, the Gaussian heat kernel estimate and the reflection principle yield
\begin{equation}
\begin{split}
\bP_{x_1}(\eqref{eq:crossing})
&\le \exp\left\{-c\frac{(\delta_{N,x}^{c_{\ref{lem:slow_cross}}{/2}}{(\log N)^3}\rad)^2}{t_2-t_1}\right\}\\
&\le {\exp\left\{-(\log N)^2\right\}\bP_{x_1}\left(\eqref{eq:stay_inside2}\right)}
\end{split}
\label{eq:fast}
\end{equation}
for all sufficiently large $N$ when $\epsilon$ is small.\\

Substituting~\eqref{eq:slow} and~\eqref{eq:fast} into~\eqref{eq:decomposition}, 
we find that
\begin{equation}
\begin{split}
& \bP\left(S_{[\tau_{B^-},\tau^{\leftarrow}_{B^-}]}\not\subset B^+, \tau_\Oi>\tau^N_x\right)\\
&\quad \le \exp\left\{-{(\log N)^2}\right\}
\sum_{t_1,t_2,x_1,x_2}
\bP\left(S_{t_1}=x_1, S_{t_2-t_1+\rad^2}=x_2, \tau_\Oi>\tau^{N+\rad^2}_x\right)\\
&\quad \le cN^{2+2d} \exp\left\{-{(\log N)^2}\right\}
\bP\left(\tau_\Oi>\tau^N_x\right),
\end{split}
\end{equation}
where in the last line, we have used that $t_1,t_2\in[0,2N]$ and $x_1,x_2\in\ball{0}{2N}$ (recall Remark~\ref{rem:macroball}). Integrating both sides with respect to $\P$, we complete the proof of~\eqref{eq:conf_middle}.
\end{proof}
\begin{remark}
\label{rem:LargeBall}
The super-polynomial rate of decay in~\eqref{eq:conf_middle} will be used later in the proof of Theorem~\ref{thm:confine}. To achieve this, as well as to counterbalance the factor $N^{2+2d}$ in the last step of the proof, we had to include $(\log N)^2$ factor in the condition for $t_2-t_1$ in Case~1 since $\delta_{N,x}$ can be as large as $\epsilon$. Then in Case~2, we needed an extra $(\log N)^3$ factor in the displacement.
This is why we included $(\log N)^3$ in~\eqref{eq:def_B+}.
\end{remark}

By the same argument, we can show that the random walk returns to $B^-$ frequently. This result will be used later to replace our condition $\{\tau_\Oi>\tau_x^N\}$ by $\{\tau_\Oi>N, S_N=x\}$ when $x$ is close to $2\rad\be_h$.
\begin{corollary}
\label{cor:rho^2}
For any $|x|\le 3\rad$,
\begin{equation}
\mu_{N,x}\left(\exists k\in [\tau_{B^-},\tau^{\leftarrow}_{B^-}-\rad^2],
S_{[k,k+\rad^2]}\cap B^-=\emptyset\right)
\le \exp\left\{-\tfrac13(\log N)^2\right\}
\label{eq:rho^2}
\end{equation}
for all sufficiently large $N$.
\end{corollary}
\begin{proof}
This can be proved in the same way as Proposition~\ref{prop:confine}. We again assume that~\eqref{eq:vacant} and~\eqref{eq:slow_cross} hold. If $S_{[k,k+\rad^2]}\cap B^-=\emptyset$, then we take $t_1$ (and $t_2$) to be the last (resp.~first) visit to $B^-$ before $k$ (resp.~after $k+\rad^2$). This probability can be bounded by $\exp\{-\delta_{N,x}^{-c_{\ref{lem:slow_cross}}}\rad^{-2}(t_2-t_1)\}$ by~\eqref{eq:slow_cross}. Comparing this with~\eqref{eq:stay_inside} and recalling that $\delta_{N,x}=\rad^{-1/5}$ when $|x|\le 3\rad$, we obtain~\eqref{eq:rho^2} as before.
\end{proof}

\section{Cost for the first and last pieces}
\label{sec:cost=dist}
In this section, we estimate the cost for the random walk to move from $0$ to $B^-$ and $B^-$ to $x$. Although it is natural to expect that they are measured by the Lyapunov distances ${\rm dist}_\beta(0,B^-)$ and ${\rm dist}_\beta(x,B^-)$, we will formulate the bound under the additional restriction that $\cent$ is fixed to be a generic point and it requires some preparation. The motivation for this formulation will be clear in Corollary~\ref{cor:confine}.

For each $|x|\le \epsilon\rad^d$ and $z\in\ball{0}{2N}$, we introduce
\begin{align}
t_{\text{out}}(x,z)&=\delta_{N,x}^{c_{\ref{prop:outsideB}}}(|z|_1+|x-z|_1)\rad^2.
\end{align}
and define a \emph{good} event by
\begin{align}
G(z)=\Bigl\{&{\Oi\cap\ball{z}{(1-\delta_{N,x}^{c_{\ref{eq:vacant}}})\rad}=\emptyset,}
\label{eq:vacant3}\\
&
B^-(z)\subset S_{[\tau_{B^-(z)},\tau^{\leftarrow}_{B^-(z)}]}\subset B^+(z)\setminus\Oi,
\label{eq:confine2}\\
& \tau^{\leftarrow}_{B^-(z)}-\tau_{B^-(z)} \ge N-t_{\text{out}}(x,z)\Bigr\}.
\label{eq:timeoutside1}
\end{align}
This event $G(z)$ morally corresponds to $\{\cent=z\}$ but is more explicit in the strategy of the random walk and the obstacle configuration.
Thanks to Propositions~\ref{prop:vacant},~\ref{prop:outsideB} and~\ref{prop:confine}, we know that
\begin{equation}
 \mu_{N,x}\left(\bigcup_{z\in B(0;2N)}G(z)\right)
\ge 1-\exp\left\{-\frac15(\log N)^2\right\}
\label{eq:G_good}
\end{equation}
for all sufficiently large $N$.

In this section, we are going to find an upper bound on
\begin{equation}
\P\otimes\bP\left(\tau_\Oi>\tau_x^N, G(z) \cond \tau_\Oi>N\right)
\end{equation}
by considering the following event that contains $\{\tau_\Oi>\tau_x^N\}\cap G(z)$:
\begin{equation}
\left\{S_{[0,\tau_{B^+(z)}]}\cap\Oi=\emptyset\right\}\cap \left\{S_{[\tau_{B^-(z)},\tau^{\leftarrow}_{B^-(z)}]}\cap\Oi=\emptyset,\eqref{eq:confine2}\right\} \cap \left\{S_{[\tau^{\leftarrow}_{B^+(z)},\tau_x^N]}\cap\Oi=\emptyset\right\},
\label{eq:3pieces}
\end{equation}
where note that we stop the first piece of random walk and the last reversed walk upon hitting $B^+(z)$, which is before hitting $B^-(z)$.
If we further specify the times $\tau_{B^-(z)}$ and $\tau^{\leftarrow}_{B^-(z)}$ and locations of the random walk at these times, then the second event is independent of the other two events and has $\P\otimes\bP$-probability not much larger than $\P\otimes\bP(\tau_\Oi>N)$ by~\eqref{eq:timeoutside1}. If the first and the third events in~\eqref{eq:3pieces} were independent, then their $\P\otimes\bP$-probabilities would decay exponentially in ${\rm dist}_\beta(0,B^-)$ and ${\rm dist}_\beta(x,B^-)$, respectively. 

However, the first and the third events in~\eqref{eq:3pieces} are not independent under $\P$ since the corresponding pieces of random walk path may overlap. For this reason, we will consider shorter pieces of the random walk path so that their survival depend on disjoint parts of environment. Let us denote the ball with respect to the Lyapunov norm in Definition~\ref{def:beta} by
\begin{equation}
 \Lball{u}{r}=\left\{v\in \Z^d\colon \beta(u-v)\le r\right\}
\end{equation}
and introduce
\begin{align}
r(z)&={\rm dist}_\beta(0,B^+(z)),\\
r(x,z)&={\rm dist}_\beta(x,B^+(z)\cup \Lball{0}{r(z)}).
\end{align}
Then by stopping the random walk and the time-reversed walk upon exiting ${\Lball{0}{r(z)}}$ and ${\Lball{x}{r(x,z)}}$, respectively, we find shortened paths which stay in disjoint sets (see Figure~\ref{fig:touching}).
\begin{figure}
 \centering
 \includegraphics{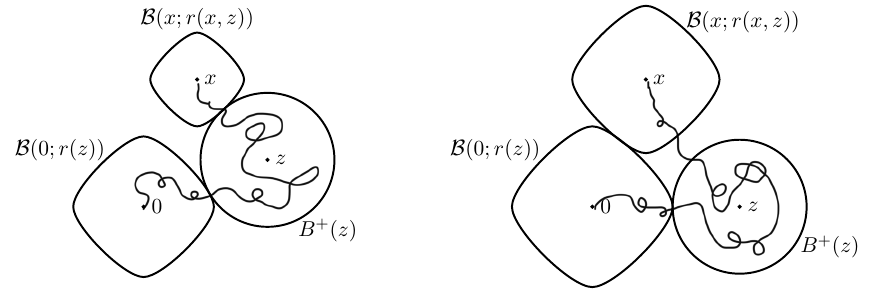}
\caption{
The convex shapes around $0$ and $x$ are the balls with respect to the Lyapunov norm with radius $r(0,z)$ and $r(x,z)$, respectively. The Euclidean ball centered at $z$ is $B^+(z)$ that has radius $(1+\delta_{N,x}^{c_{\ref{lem:slow_cross}}{/2}}{(\log N)^3})\rad$. In the left picture, $\Lball{x}{r(x,z)}$ touches $B^+(z)$ while in the right picture it touches $\Lball{0}{r(z)}$.}
\label{fig:touching}
\end{figure}
We are now ready to state the main result of this section.
\begin{proposition}
\label{prop:z_fixed}
When $\epsilon>0$ is small depending on $d$ and $p$, $|x|\le \epsilon\rad^d$ and $z\in\ball{0}{2N}$, 
\begin{equation}
\begin{split}
&\P\otimes\bP\left(\tau_\Oi>\tau_x^N, G(z)\cond \tau_\Oi>N\right)\\
&\quad \le \exp\left\{-(1-\epsilon)(r(z)+r(x,z))+\delta_{N,x}^{c_{\ref{prop:outsideB}}/2}(|z|_1+|x-z|_1) +(\log N)^2\right\}
\end{split}
\label{eq:z_fixed}
\end{equation}
for all sufficiently large $N$.
\end{proposition}
\begin{proof}
Let us fix $k,l\in[0,2N]$ satisfying
\begin{equation}
\begin{split}
 l-k &\ge N-t_{\text{out}}(x,z)\\
 &=N-\delta_{N,x}^{c_{\ref{prop:outsideB}}}(|z|_1+|x-z|_1)\rad^2,
\end{split}
\label{eq:timeoutside}
\end{equation}
and define
\begin{equation}
{G(z;k,l)=\left\{\Oi\cap\ball{z}{(1-\delta_{N,x}^{c_{\ref{eq:vacant}}})\rad}=\emptyset,
B^-(z)\subset S_{[k,l]}\subset B^+(z)\setminus\Oi\right\}}.
\end{equation}
We further introduce $x_1,x_2\in B^-(z)$ and start by rewriting
\begin{equation}
\begin{split}
&\P\otimes\bP\left(\tau_\Oi>\tau_x^N, {\tau_{B^-}=k, S_k=x_1, G(z), \tau_{B^-}^\leftarrow=l,,S_l=x_2}\right)\\
&\quad = \P\otimes\bP\Bigl(\tau_{B^-(z)}=k, {S_k=x_1}, S_{[0,k]}\cap\Oi=\emptyset, {G(z;k,l)},\\
&\hspace{120pt}\tau^{\leftarrow}_{B^-(z)}=l,{S_l=x_2},S_{[l,\tau_x^N]}\cap\Oi=\emptyset\Bigr).
\end{split}
\label{eq:zkl_fixed}
\end{equation}
We will take a sum over $k,l,x_1,x_2$ in the end. We are going to estimate the costs for the three pieces $S_{[0,{\tau_{B^+(z)}}]}$, $S_{[k,l]}$ and $S_{[{\tau^\leftarrow_{B^+(z)}},\tau_x^N]}$ to avoid $\Oi$ separately. Noting that $\tau_{B^+(z)}>\tau_{\Lball{0}{r(z)}^c}$, we have
\begin{equation}
 \bP\left({\tau_{B^-(z)}=k,S_k=x_1},S_{[0, {k}]}\cap\Oi=\emptyset\right)
\le \bP\left(\tau_\Oi>\tau_{\Lball{0}{r(z)}^c}\right)
\end{equation}
and similarly, by considering the time-reversed random walk,
\begin{equation}
\bP\left({\tau^{\leftarrow}_{B^-(z)}=l,S_l=x_2},S_{[{l}, \tau_x^N]}\cap\Oi=\emptyset\right)
\le \bP_x\left(\tau_\Oi>\tau_{\Lball{x}{r(x,z)}^c}\right).
\end{equation}
Since $\Lball{x}{r}$, $\Lball{0}{r(z)})$ and $B^+(z)$ are disjoint, the right-hand sides of the above two inequalities and
$\bP({G(z;k,l)\mid S_k=x_1,S_l=x_2})$
are independent under $\P$. Therefore it follows from~\eqref{eq:zkl_fixed} that
\begin{equation}
\begin{split}
&\P\otimes\bP\left(\tau_\Oi>\tau_x^N, {\tau_{B^-}=k, S_k=x_1, G(z), \tau_{B^-}^\leftarrow=l,,S_l=x_2}\right)\\
&\quad \le \P\otimes\bP\left(\tau_\Oi>\tau_{\Lball{0}{r(z)}^c}\right)
\P\otimes\bP\bigl({G(z;k,l)\mid S_k=x_1,S_l=x_2}\bigr)
\P\otimes\bP_x\left(\tau_\Oi>\tau_{\Lball{x}{r(x,z)}^c}\right).
\end{split}
\label{eq:separated}
\end{equation}
We begin with the first and third factors. From~\cite[(0.5)]{Szn95b1} and the union bound, it follows that for any $\epsilon\in(0,1)$,
\begin{align}
\P\otimes\bP\left(\tau_\Oi>\tau_{\Lball{0}{r(z)}^c}\right)
&\le cN^d \exp\left\{-(1-\epsilon)r(z)\right\},\label{eq:first}\\
\P\otimes\bP_x\left(\tau_\Oi>\tau_{\Lball{x}{r(x,z)}^c}\right)
&\le cN^d\exp\left\{-(1-\epsilon)r(x,z)\right\},\label{eq:last}
\end{align}
where $cN^d$ is a crude upper bound on $|\partial \Lball{0}{r(z)}|$ and $|\partial \Lball{x}{r(x,z)}|$.
For the second factor in~\eqref{eq:separated}, we use~\eqref{eq:timeoutside} and the local central limit theorem to obtain
\begin{equation}
\begin{split}
 \bP\left({G(z;k,l) \mid S_k=x_1, S_l=x_2}\right)
& \le \bP_{x_1}\left(\tau_\Oi>N-t_{\text{out}}(x,z)\cond{S_{l-k}=x_2}\right)\\
& \le N^c \bP_{x_1}\left(\tau_\Oi>N-t_{\text{out}}(x,z)\right).
\end{split}
\label{eq:mid_piece}
\end{equation}
We further add a piece of random walk loop satisfying
\begin{align}
S_0=S_{t_{\text{out}}(x,z)}=x_1, S_{[0,t_{\text{out}}(x,z)]}\subset \ball{z}{(1-\delta_{N,x}^{c_{\ref{eq:vacant}}})\rad}
\end{align}
in order to recover $\tau_\Oi>N$. Due to~\eqref{eq:vacant3}, the additional cost can be controlled by the random walk estimate~\eqref{eq:ballHK} and the right-hand side of~\eqref{eq:mid_piece} is bounded by
\begin{equation}
N^{{c}}\exp\left\{c\delta_{N,x}^{c_{\ref{prop:outsideB}}}(|z|_1+|x-z|_1)\right\}\bP_{x_1}\left(\tau_\Oi>N\right).
\end{equation}
By the translation invariance of $\P$ and the union bound over $x_1\in\ball{0}{2N}$, it follows that
\begin{equation}
\begin{split}
&\P\otimes\bP\left(G(z;k,l)\cond S_k=x_1,S_l=x_2\right)\\
&\quad\le  N^{{c}}\exp\left\{{C_1}\delta_{N,x}^{c_{\ref{prop:outsideB}}}(|z|_1+|x-z|_1)\right\}
\P\otimes\bP\left(\tau_\Oi>N\right).
\end{split}
\label{eq:middle}
\end{equation}
Substituting~\eqref{eq:first},~\eqref{eq:last} and~\eqref{eq:middle} into~\eqref{eq:separated}, we arrive at
\begin{equation}
\begin{split}
&\P\otimes\bP\left(\tau_\Oi>\tau_x^N, {\tau_{B^-}=k, S_k=x_1, G(z), \tau_{B^-}^\leftarrow=l,,S_l=x_2}\cond \tau_\Oi>N\right)\\
&\quad \le N^{{c}}\exp\left\{-(1-\epsilon)(r(z)+r(x,z))+{C_1}\delta_{N,x}^{c_{\ref{prop:outsideB}}}
(|z|_1+|x-z|_1)\right\}.
\end{split}
\label{eq:cost_sep}
\end{equation}
Summing over $k,l\le 2N$ and $x_1,x_2\in B^-(z)$, and choosing $\epsilon$ so small that ${C_1}\delta_{N,x}^{c_{\ref{prop:outsideB}}}\le \delta_{N,x}^{c_{\ref{prop:outsideB}}/2}$, we obtain~\eqref{eq:z_fixed}.
\end{proof}
Proposition~\ref{prop:z_fixed} measures not only the cost for the random walk to visit $B^-(z)$, but also the cost as $z$ varies. In the following corollary, we use it to show that if $|x|$ is close to $2\rad$, then $z$ must be near $\frac12 x$. In addition, we show that the whole random walk path is confined in a ball slightly larger than $B^+$. This will be used in the proof of~\eqref{eq:confine}.
\begin{corollary}
\label{cor:confine}
Let $h\in\R^d$ and $\be_h=h/|h|$. When $\epsilon>0$ is small depending on $d$ and $p$ and $x\in\ball{2\rad\be_h}{\epsilon\rad}$,
\begin{equation}
\begin{split}
& \mu_{N,x}\left({\bigcup_{z\in\ball{\rad\be_h}{{\epsilon^{1/4}}\rad}}}G(z) \cap \left\{ S_{[0,\tau_x^N]} \subset \ball{\rad\be_h}{(1+{\epsilon^{1/5}})\rad}\right\}\right)\\
&\quad \ge 1- \exp\left\{-\tfrac14(\log N)^2\right\}
\end{split}
\label{eq:cor_confine}
\end{equation}
for all sufficiently large $N$.
\end{corollary}
\begin{proof}
Note first that when $x\in B({2\rad\be_h};{\epsilon\rad})$, Proposition~\ref{prop:LDPlower} yields
\begin{equation}
\P\otimes\bP\left(\tau_{\Oi}>\tau_x^N \cond \tau_\Oi>N\right)
\ge \exp\left\{-c\epsilon\rad\right\}.
\label{eq:pf_|x|=2rad}
\end{equation}

Let us prove that we can discard $G(z)$ with $z$ not close to $\rad\be_h$.
For any $x\in\ball{2\rad\be_h}{\epsilon\rad}$ and $z\not\in B({\rad\be_h};{{\epsilon^{1/4}}\rad})$, we have
\begin{equation}
 r(z)+r(x,z) \ge c\epsilon^{1/2}\rad.
\end{equation}
Substituting this into~\eqref{eq:z_fixed} and comparing with~\eqref{eq:pf_|x|=2rad}, we find that for any $z\not\in B({\rad\be_h};{{\epsilon^{1/4}}\rad})$,
\begin{equation}
\mu_{N,x}\left(G(z)\right)
\le \exp\left\{-c{\epsilon}^{1/2}\rad\right\}.
\end{equation}

Next we prove the confinement part. By Proposition~\ref{prop:confine} and what we have just proved, we may assume that
\begin{equation}
G(z)\text{ holds for some }z\in \ball{\rad\be_h}{{\epsilon^{1/4}}\rad}\text{ and } S_{[\tau_{B^-(z)},\tau_{B^-(z)}^{\leftarrow}]}\subset B^+(z).
\end{equation}
Therefore, the random walk can exit $B({\rad\be_h};{(1+{\epsilon^{1/5}})\rad})$ only during the time interval $[0,\tau_{B^-(z)}]$ or $[\tau_{B^-(z)}^{\leftarrow},\tau_x^N]$. We first consider the former case. In this case, we use the strong Markov property at the first exit time from $B({\rad\be_h};{(1+\epsilon^{1/5})\rad})$. Starting from the exit time, we repeat the proof of Proposition~\ref{prop:z_fixed}. Then the cost for the first piece of the random walk becomes the Lyapunov distance between $B({\rad\be_h};{(1+\epsilon^{1/5})\rad})^c$ and $B^+(z)$, which is larger than $c\epsilon^{1/5}\rad$. Thus it follows for any $z\in B({\rad\be_h};{{\epsilon^{1/4}}\rad})$ that
\begin{equation}
\begin{split}
& \P\otimes\bP\left(\tau_{\Oi}>\tau_x^N, G(z), S_{[0,\tau_{B^-}]}\not \subset \ball{\rad\be_h}{(1+{\epsilon^{1/5}})\rad}\cond \tau_\Oi>N\right)\\
& \quad \le \exp\left\{-c{\epsilon^{1/5}}\rad\right\}
\end{split}
\end{equation}
and comparing with~\eqref{eq:pf_|x|=2rad}, we conclude that
\begin{equation}
 \mu_{N,x}\left(G(z), S_{[0,\tau_{B^-}]}\not \subset \ball{\rad\be_h}{(1+{\epsilon^{1/5}})\rad}\right)
\le \exp\left\{-c{\epsilon^{1/5}}\rad\right\}.
\end{equation}
By the same argument, we get the same bound for the probability that the random walk exits from $B({\rad\be_h};{(1+{\epsilon^{1/5}})\rad})$ during the time interval $[\tau_{B^-(z)}^{\leftarrow},\tau_x^N]$ and we are done.
\end{proof}
\begin{remark}
\label{rem:unbiased1}
As long as $|x|\le (2+\epsilon)\rad$, we can follow the same argument as above to prove that the random walk path $S_{[0,N]}$ is confined in some $B(z;{(1+{\epsilon^{1/5}})\rad})$ which contains both 0 and $x$, and also $S_{[0,N]}$ covers a slightly smaller ball with the same center.
\end{remark}

\section{Proof of the upper bound in Theorem~\ref{thm:LDP}}
\label{sec:LDPupper}
In this section, we prove the following proposition, which in particular implies the upper bound in Theorem~\ref{thm:LDP}. Combined with the lower bound proved in Section~\ref{sec:LDPlower}, it completes the proof of Theorem~\ref{thm:LDP}.
\begin{proposition}
\label{prop:LDPupper}
 There exists $c_{\ref{prop:LDPupper}}>0$ such that when $\epsilon>0$ is small depending on $d$ and $p$ and $(2+\epsilon)\rad\le |x|\le \epsilon\rad^d$,
 \begin{equation}
 \P\otimes\bP\left(S_N=x \cond \tau_\Oi>N \right)
 \le \exp\left\{-(1-\epsilon^{c_{\ref{prop:LDPupper}}}){\rm dist}_\beta(x,\ball{0}{2\rad})\right\}
 \label{eq:LDPupper}
 \end{equation}
for all sufficiently large $N$.
\end{proposition}
\begin{proof}
By the fact $\{S_N=x,\tau_\Oi>N\}\subset \{\tau_\Oi >\tau_x^N\}$ and~\eqref{eq:G_good}, we have
\begin{equation}
 \P\otimes\bP\left(S_N=x,\tau_\Oi>N \right)
\le (1+o(1))\P\otimes\bP\left(\tau_\Oi>\tau_x^N, {\bigcup_{z\in B(0;2N)}G(z)}\right)
\end{equation}
as $N\to \infty$. Therefore, it suffices to bound $\P\otimes\bP(\tau_\Oi>\tau_x^N, \bigcup_{z\in B(0;2N)}G(z) \mid \tau_\Oi>N)$ by the right-hand side of~\eqref{eq:LDPupper}. We can bound this probability by using Proposition~\ref{prop:z_fixed} and the union bound as follows:
\begin{equation}
\begin{split}
&\P\otimes\bP\left(\tau_\Oi>\tau_x^N, {\bigcup_{z\in B(0;2N)}G(z)} \cond \tau_\Oi>N \right)\\
&\quad \le \sum_{z\in B(0;2N)}\P\otimes\bP\left(\tau_\Oi>\tau_x^N, G(z)\cond \tau_\Oi>N \right)\\
&\quad \le cN^d \exp\left\{-\inf_{z\in\ball{0}{2N}}\left[(1-\epsilon)(r(z)+r(x,z))-\delta_{N,x}^{c_{\ref{prop:outsideB}}/2}(|z|_1+|x-z|_1)+(\log N)^2\right]\right\}.
\end{split}
\label{eq:LDP+error}
\end{equation}

Let us first consider the case $|x|\le \rad^{d-1/2}$ and prove that for any $z\in \ball{0}{2N}$,
\begin{equation}
 r(z)+r(x,z) \ge {\rm dist}_\beta(x,\ball{0}{2\rad})+o(\rad)
\label{eq:r+r>dist}
\end{equation}
as $N\to\infty$. We may assume that $r(x,z)={\rm dist}_\beta(x,B^+(z))$, that is, we are in the situation of the left picture in Figure~\ref{fig:touching}. Otherwise, we can decrease the left-hand side of~\eqref{eq:r+r>dist} by moving $z$ to the point where $\Lball{0}{r(z)}$ touches $\Lball{x}{r(x,z)}$. Then for any $r>0$, we have
\begin{equation}
\begin{split}
&\min\left\{r(z)+{\rm dist}_\beta(x,\ball{z}{\rad})\colon z\in\ball{0}{2N}, r(z)=r\right\}\\
&\quad \ge \min\left\{\beta(u)+\beta(x-(u+v))\colon u\in \partial\Lball{0}{r}, v\in\ball{0}{2(1+\delta_{N,x}^{c_{\ref{lem:slow_cross}}{/2}}{(\log N)^3}))\rad}\right\}
\end{split}
\end{equation}
by choosing $u$ and $u+v$ so that $\beta(u)={\rm dist}_\beta(0,B^+(z))$ and $\beta(x-(u+v))={\rm dist}_\beta(x,B^+(z))$, respectively. Since $\beta(\cdot)$ is a norm, the above is further bounded from below by
\begin{equation}
\min\left\{\beta(x-v)\colon v\in\ball{0}{2(1+\delta_{N,x}^{c_{\ref{lem:slow_cross}}{/2}}{(\log N)^3})\rad}\right\}
 \ge {\rm dist}_\beta(x,\ball{0}{2\rad})+o(\rad)
\end{equation}
as $N\to\infty$, in the case $|x|\le \rad^{d-1/2}$. Since $r>0$ was arbitrary, this proves~\eqref{eq:r+r>dist}.

Next, by the assumption $|x|\ge (2+\epsilon)\rad$ and~\eqref{eq:r+r>dist}, we have $r(x)+r(x,z)\ge c\epsilon|x|$ and hence for $|z|\le 2|x|$,
\begin{equation}
|z|_1+|x-z|_1 \le c\epsilon^{-1} (r(z)+r(x,z)).
\label{eq:r>||}
\end{equation}
This bound remains valid for $|z|> 2|x|$ since $r(z)\ge c|z|$ and $|z|_1+|x-z|_1\le 3|z|$ in this case. Substituting~\eqref{eq:r>||} and~\eqref{eq:r+r>dist} into~\eqref{eq:LDP+error}, we obtain the desired bound since in the case $|x|\le \rad^{d-1/2}$, we have $\lim_{N\to\infty}\delta_{N,x}=0$.

In the other case $|x|>\rad^{d-1/2}$, it is easily seen that the size of $B^+(z)$ is negligible compared with $r(x)+r(x,z)$. Then it follows that
\begin{equation}
\begin{split}
r(z)+r(x,z) &=(\beta(z)+\beta(x-z))(1+o(1))\\
&= {\rm dist}_\beta(x,\ball{0}{2\rad})(1+o(1))
\end{split}
\label{eq:r+r>dist2}
\end{equation}
and that
\begin{equation}
 |z|+|x-z|\le c\left(r(z)+r(x,z)\right)
\end{equation}
as $N\to\infty$. Using this bound instead of~\eqref{eq:r+r>dist}, we can complete the proof as before.
\end{proof}
\begin{remark}
\label{rem:unbiased2}
In the case $h=0$, Theorem~\ref{thm:LDP} shows that $\mu_N(|S_N|\le (2+\epsilon)\rad)\to 1$ as $N\to\infty$. Combining this with Remark~\ref{rem:unbiased1}, we get a proof of~\eqref{eq:SBP}.
\end{remark}

\section{Proof of Theorem~\ref{thm:confine}}
\label{sec:thm_drift}
In this section we prove Theorem~\ref{thm:confine}.

\begin{proof}[Proof of Theorem~\ref{thm:confine}]
Let us start by proving~\eqref{eq:LLN}. We can deduce it from Proposition~\ref{prop:LDPupper} and the large deviation results in~\cite{Szn95b1,Szn95b2} by a standard exponential tilting argument~\cite[Theorem~II.7.2]{Ell85}. But we provide a more direct argument which elucidates the role of the assumption $\beta^*(h)<1$.

Let us recall that by~\cite[Theorem~2.2]{Szn95b2},
\begin{equation}
\lim_{N\to\infty} \mu_N^h\left(|S_N|\le \epsilon\rad^d\right)= 1.
\label{eq:Szn95b2}
\end{equation}
Choosing $x=2\rad\be_h$ in~\eqref{eq:LDPlower}, we find the following lower bound on the partition function of $\mu^h_N$:
\begin{equation}
\begin{split}
\E\otimes\bE\left[e^{ \langle h, S_N \rangle}\colon \tau_\Oi>N\right]
&\ge \E\otimes\bE\left[e^{\langle h, S_N \rangle}\colon \tau_\Oi>N, S_N=2\rad\be_h\right]\\
&\ge e^{(2|h|-c_{\text{\upshape\ref{prop:LDPlower}}}\epsilon)\rad}\P\otimes \bP(\tau_\Oi>N).
\end{split}
\label{eq:pf_drift}
\end{equation}
On the other hand, we have
\begin{equation}
\E\otimes\bE\left[e^{ \langle h, S_N \rangle}\colon \tau_\Oi>N, S_N=x\right]
\le e^{(2|h|-c\epsilon^{1/2})\rad}\P\otimes\bP\left(\tau_\Oi>N\right)
\label{eq:bias_x_fixed}
\end{equation}
for all sufficiently small $\epsilon>0$ and $x\in\ball{0}{2N}$ satisfying either of the following conditions:
\begin{enumerate}
 \item $\langle h,x \rangle < (2|h|-\epsilon^{1/2})\rad$, in which case we simply drop the constraints $S_N=x$ and
 \item $|x|>(2+\epsilon^{1/2})\rad$, in which case by Proposition~\ref{prop:LDPupper}, the subcriticality assumption $\beta^*(h)<1$, and taking $y\in\ball{0}{2\rad}$ such that $\beta(x-y)={\rm dist}_\beta(x,\ball{0}{2\rad})$,
\begin{equation}
\begin{split}
\langle h,x \rangle-{\rm dist}_\beta(x,\ball{0}{2\rad})
&\le 2\rad |h|+\langle h,x-y \rangle-\beta(x-y)\\
&\le 2\rad |h|+\beta(x-y) (\beta^*(h)-1).
\end{split}
\end{equation}
\end{enumerate}
Comparing~\eqref{eq:bias_x_fixed} with~\eqref{eq:pf_drift} and summing over $x\in\ball{0}{2N}$, we obtain
\begin{equation}
\lim_{N\to\infty} \mu_N^h\left(\langle h,S_N \rangle \ge (2|h|-\epsilon^{1/2})\rad \text{ and } |S_N|\le (2+\epsilon^{1/2})\rad\right)
=1
\label{eq:arc}
\end{equation}
for sufficiently small $\epsilon>0$. Since
\begin{equation}
\left\{x\colon\langle h,x \rangle \ge (2{|h|}-\epsilon^{1/2})\rad\right\}\cap \ball{0}{(2+\epsilon^{1/2})\rad} \subset B(2\rad\be_h;c\epsilon^{1/4}\rad)
\end{equation}
for small $\epsilon$ (see Figure~\ref{fig:Tangent}), the proof of~\eqref{eq:LLN} is completed.
\begin{figure}
 \centering
 \includegraphics{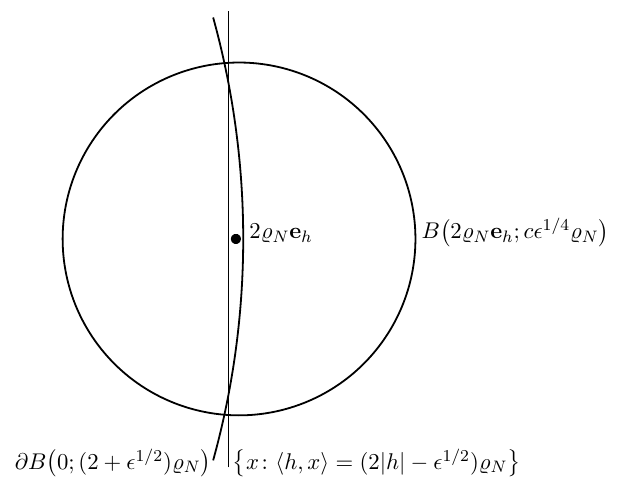}
\caption{The balls and hyperplane appearing in the proof of~\eqref{eq:LLN}.}
\label{fig:Tangent}
\end{figure}

Next we turn to the proof of~\eqref{eq:confine}.
Since we have already proved~\eqref{eq:LLN}, we have
\begin{equation}
\begin{split}
\mu_N^h(A) = \sum_{x\in\ball{2\rad \be_h}{\epsilon\rad}} \mu_N^h\left(S_N=x,A\right)+o(1)
\end{split}
\label{eq:endpoint_wise}
\end{equation}
for any event $A$ as $N\to\infty$. Let us now introduce the \emph{pinned} measure
\begin{equation}
 \tilde{\mu}_{N,x}(\cdot)=\P\otimes\bP(\cdot\mid S_N=x, \tau_\Oi>N).
\label{eq:pinneddef}
\end{equation}
We assume the following lemma for the moment.
\begin{lemma}
 \label{lem:pinned}
There exists $c_{\ref{lem:pinned}}>0$ such that for any $|x|\le 3\rad$ and any event $A$,
\begin{equation}
\tilde{\mu}_{N,x}(A) \le N^{c_{\ref{lem:pinned}}}\mu_{N,x}(A)
\label{eq:pinned}
\end{equation}
for all sufficiently large $N$.
\end{lemma}
This lemma implies that the summand in~\eqref{eq:endpoint_wise} can be estimated as
\begin{equation}
\begin{split}
\mu_N^h\left(S_N=x,A\right)
&=e^{\langle h, x \rangle}\tilde\mu_{N,x}\left(A\right)
 \frac{\P\otimes\bP\left(S_N=x, \tau_\Oi > N \right)}{\E\otimes\bE\left[e^{\langle h,S_N \rangle}\colon \tau_\Oi>N\right]}\\
&\le e^{\langle h, x \rangle}\tilde\mu_{N,x}\left(A\right)
 \frac{\P\otimes\bP\left(S_N=x, \tau_\Oi > N \right)}{\E\otimes\bE\left[e^{\langle h,S_N \rangle}\colon S_N=x, \tau_\Oi>N\right]}\\
&= \tilde\mu_{N,x}\left(A\right)\\
&\le N^{c_{\ref{lem:pinned}}}\mu_{N,x}\left(A\right).
\end{split}
\label{eq:mu^h<mu_x}
\end{equation}
Now we choose $A$ to be the event
\begin{equation}
\left({\bigcup_{z\in\ball{\rad\be_h}{{\epsilon^{1/4}}\rad}}}G(z) \cap \left\{ S_{[0,\tau_x^N]} \subset \ball{\rad\be_h}{(1+{\epsilon^{1/5}})\rad}\right\}\right)^c
\end{equation}
Since the $\mu_{N,x}$ probability of this event decays super-polynomially by Corollary~\ref{cor:confine}, so does the left-hand side of~\eqref{eq:mu^h<mu_x}. Coming back to~\eqref{eq:endpoint_wise} and recalling~\eqref{eq:confine2} in the definition of $G(z)$, we conclude that
\begin{equation}
\mu_N^h\left(\ball{\rad\be_h}{\left(1-\epsilon^{1/3}\right)\rad}\subset
S_{[0,N]}\subset \ball{\rad\be_h}{\left(1+\epsilon^{1/3}\right)\rad}
\right)\to 1
\end{equation}
as $N\to \infty$. Since this holds for all sufficiently small $\epsilon>0$, we complete the proof of~\eqref{eq:confine}.
\end{proof}
\begin{proof}
[Proof of Lemma~\ref{lem:pinned}]
We are going to prove
\begin{equation}
 \P\otimes\bP\left(\tau_\Oi>\tau_x^N\right)
\le N^c \P\otimes\bP\left(S_N=x, \tau_\Oi>N\right).
\label{eq:pf_pinned}
\end{equation}
From this and $\{S_N=x, \tau_\Oi>N\}\subset\{\tau_\Oi>\tau_x^N\}$, we can deduce~\eqref{eq:pinned}. The proof of~\eqref{eq:pf_pinned} relies on a path switching argument: we show that a path with $\tau_\Oi>\tau_x^N$ can be shortened to satisfy $S_N=x$ and $\tau_\Oi>N$ without paying too much cost.
To this end, let us define a good event by
\begin{align}
G'=\Bigl\{&\Oi\cap \ball{\cent}{(1-\delta_{N,x}^{c_{\ref{prop:vacant}}})\rad}=\emptyset,\label{eq:vacant8}\\
&\tau_{B^-} \vee \left(\tau_x^N-\tau_{B^-}^{\leftarrow}\right) \le \epsilon N,\label{eq:short}\\
&S_{[\tau_{B^-}, \tau_{B^-}^{\leftarrow}]} \subset B^+,\label{eq:confine8}\\
&\forall k\in[\tau_{B^-},\tau_{B^-}^{\leftarrow}-\rad^2], S_{[k,k+\rad^2]}\cap B^-\neq\emptyset.\Bigr\}\label{eq:return}
\end{align}
Under the assumption $|x|\le \epsilon\rad^d$, by Propositions~\ref{prop:vacant}~,~\ref{prop:outsideB} and~\ref{prop:confine} and Corollary~\ref{cor:rho^2}, we have
\begin{equation}
 \P\otimes\bP\left(\tau_\Oi>\tau_x^N\right)
= (1+o(1)) \P\otimes\bP\left(\tau_\Oi>\tau_x^N, G'\right)
\end{equation}
as $N\to\infty$, and hence it suffices to bound the right-hand side.

Now suppose that $\tau_x^N=N+l<\tau_\Oi$ and $G'$ holds, where we may assume $l\le N$ by Corollary~\ref{cor:toolong}. Then from~\eqref{eq:short} and~\eqref{eq:return}, it follows that there exists $m\in [l+\rad^2, l+2\rad^2]$ such that $S_{\tau_{B^-}+m}\in B^-(z)$. We make a case distinction according to $\tau_{B^-}=n$ ($0\le n\le \epsilon N$) and use the Markov property at time $n$ and $n+m$ to obtain
\begin{equation}
\begin{split}
&\P\otimes\bP\left(\tau_\Oi>\tau_x^N=N+l,G'\right)\\
&\quad \le \sum_{n\le \epsilon N}\sum_{m\in [l+\rad^2, l+2\rad^2]}\sum_{y,z\in B^-}\E\left[p^{\Z^d\setminus\Oi}_n(0,y)p^{B^+}_m(y,z)p^{\Z^d\setminus\Oi}_{N+l-m-n}(z,x)\colon \text{\eqref{eq:vacant8}}\right],
\end{split}
\label{eq:lmnyz}
\end{equation}
where $p^U_n(x,y)$ stands for the transition probability of the random walk killed upon existing from $U$ (see below~\eqref{eq:SG2}). 
We are going to shorten the time in $p^{B^+}_m(y,z)$.
Since $p^{B^+}_m(y,z)\le 1$, $p^{B^-}_{m-l}(y,z)\ge c\rad^{-d-1}$ by~\cite[Proposition~6.9.4]{LL10}, and $B^-\subset \Z^d\setminus\Oi$ by~\eqref{eq:vacant8}, we have 
\begin{equation}
\begin{split}
 p^{B^+}_m(y,z)&\le N^c p^{B^-}_{m-l}(y,z)\\
&\le N^c p^{\Z^d\setminus\Oi}_{m-l}(y,z).
\end{split}
\end{equation}
Substituting this into~\eqref{eq:lmnyz} and summing over $l\le N$, we obtain
\begin{equation}
\P\otimes\bP\left(\tau_\Oi>\tau_x^N,G'\right)
\le N^c \P\otimes\bP\left(S_N=x, \tau_\Oi>N\right)
\end{equation}
and we are done.
\end{proof}

\appendix
\section{Proof of Lemmas~\ref{lem:almost} and~\ref{lem:slow_cross} by the Method of Enlargement of Obstacles}
\label{sec:MEO}
If the reader is familiar with the method of enlargement of obstacles developed in~\cite{Szn97a,Szn98}, or its discrete version in~\cite{BAR00}, then it is rather easy to prove Lemmas~\ref{lem:almost} and~\ref{lem:slow_cross} by adapting the proofs of~\cite[Proposition~2]{Pov99} and~\cite[Lemma~1]{Pov99}, respectively. In fact, a statement similar to Lemma~\ref{lem:almost} appeared in~\cite[Remark~1.4]{Szn95b2}.

In~\cite{Pov99}, the method of enlargement of obstacles is used to construct a certain set $\mathscr{U}_{\rm cl}$ which is almost free from obstacles and also any point outside is well surrounded by obstacles. We shall recall these properties more precisely in the proofs.

In this alternative argument, we need to change the exponent $1/5$ in~\eqref{eq:def_delta} to another $\chi>0$ depending only on $d$ and $p$, to be determined later:
\begin{equation}
\delta_{N,x}=\rad^{-\chi}\vee ({|x|}/{\rad^{d}}).
\label{eq:def_delta2}
\end{equation}
\begin{remark}
When one compares the following argument with that in~\cite{Pov99}, it is important to keep in mind that space is scaled by the factor $\rad^{-1}$ in~\cite{Pov99}.
\end{remark}
\begin{proof}[Proof of Lemma~\ref{lem:almost}]
It is shown in~\cite[(46), (49), (52) and (56)]{Pov99} that there exist $c_4>1$, $\alpha_1,\alpha_2>0$ and a random set $\mathscr{U}_{\rm cl}\subset \ball{0}{2N}$ such that
\begin{equation}
\begin{split}
 &\P\otimes\bP\left(|\mathscr{U}_{\rm cl}|\log \tfrac{1}{p}+N\lambda_{\mathscr{U}_{\rm cl}}>N^{\frac{d}{d+2}}\left(c(d,p)+c_4N^{-\frac{\alpha_1}{d+2}}\right), \tau_\Oi>N\right)\\
&\quad \le \exp\left\{-N^{\frac{d}{d+2}}\left(c(d,p)+{c_4}N^{-\frac{\alpha_1}{d+2}}\right)+N^{\frac{d-\alpha_1}{d+2}}\right\},\label{eq:MEO_var}
\end{split}
\end{equation}
and
\begin{equation}
\P\otimes\bP\left(|\Oi\cap\mathscr{U}_{\rm cl}| \ge t^{\frac{d-\alpha_2}{d+2}}\right)
\le\exp\left\{-\tfrac32 c(d,p) N^{\frac{d}{d+2}}\right\}.
\label{eq:MEO_vol}
\end{equation}
(The parameter $\alpha_2$ corresponds to $\kappa-d\alpha_0$ in~\cite{Pov99}.)
In~\eqref{eq:MEO_var}, the term $c_4N^{-\frac{\alpha_1}{d+2}}$ in fact directly inherits from the first line to the second, and this bound remains valid if we replace it by $2\beta\left(\be_x\right)\delta_{N,x}$, and $\tau_\Oi>N$ by $\tau_\Oi>\tau_x^N$, if we choose $\chi<\alpha_1$ in~\eqref{eq:def_delta2}. Then we can repeat the argument in~\cite[Proposition~2]{Pov99} to see that on the event
\begin{equation}
\left\{|\mathscr{U}_{\rm cl}|\log \tfrac{1}{p}+N\lambda_{\mathscr{U}_{\rm cl}}
\le N^{\frac{d}{d+2}}\left(c(d,p)+2\beta\left(\be_x\right)\delta_{N,x}\right)\right\},
\end{equation}
there exists a ball $\ball{\cent}{\rad}$ and $c_5>0$ such that
\begin{equation}
|\mathscr{U}_{\rm cl}\triangle \ball{\cent}{\rad}|
<c_5(\beta\left(2\be_x\right)\delta_{N,x})^{\frac{1}{32}}\rad^d
\label{eq:symmdiff1}
\end{equation}
by the quantitative Faber--Krahn inequality in~\cite{BDV15}.
\begin{remark}
As in Section~\ref{appendix}, we need to enlarge $\mathscr{U}_{\rm cl}$ to $\boldsymbol{\mathscr{U}}_{\rm cl}^+$ in order to apply the quantitative Faber--Krahn inequality in~\cite{BDV15}. But we do not need to worry about the increase of volume since $\mathscr{U}_{\rm cl}$ is defined as a union of cubes of the form $\rad^{1-\gamma} (q+[0,1)^d)$ with $\gamma\in(0,1)$ and $q\in\Z^d$.
\end{remark}
Combining the above consideration with~\eqref{eq:pf} and~\eqref{eq:MEO_vol}, we conclude that if $c<\frac{1}{32}\wedge \frac{\alpha_2}{\chi}$ so that
\begin{equation}
\delta_{N,x}^c\rad^d
\ge c_5(2\beta\left(\be_x\right)\delta_{N,x})^{\frac{1}{32}}\rad^d
+N^{\frac{d-\alpha_2}{d+2}}
\end{equation}
for sufficiently large $N$, then
\begin{equation}
\begin{split}
&\mu_{N,x}\left(|\Oi\cap \ball{\cent}{\rad}|
 \ge \delta_{N,x}^c\rad^d\right)\\
&\quad\le \mu_{N,x}\left(|\mathscr{U}_{\rm cl}\triangle \ball{\cent}{\rad}|
 \ge c_5(2\beta\left(\be_x\right)\delta_{N,x})^{\frac{1}{32}}\rad^d\right)\\
&\quad < \exp\left\{\beta\left(\be_x\right)(|x|-2\rad)+\epsilon \delta_{N,x}N^{\frac{d}{d+2}}-2\beta\left(\be_x\right)\delta_{N,x}N^{\frac{d}{d+2}}\right\}\\
&\quad < \exp\left\{-\frac{1}{2}\beta(\be_x)\delta_{N,x}N^{\frac{d}{d+2}}\right\},
\end{split}
\label{eq:almost_proof}
\end{equation}
which implies~\eqref{eq:almost}.
\end{proof}
\begin{proof}[Proof of Lemma~\ref{lem:slow_cross}]
By~\eqref{eq:almost_proof}, it suffices to prove that on the event~\eqref{eq:symmdiff1}, there exists $c>0$ such that
\begin{equation}
\label{eq:forest}
\lambda_{\ball{0}{2N}\setminus (B^-\cup\Oi)}
\ge \delta_{N,x}^{-c}\rad^{-2}.
\end{equation}
Indeed, this and a well-known semigroup bound~\cite[(2.21)]{Kon16} imply that when $N$ is sufficiently large, for any $y\in\Z^d$ and $\delta_{N,x}^{-c/3} \rad^{-2}\le t\le 2N$,
\begin{equation}
\begin{split}
 \bP_y\left(S_{[0,t]}\cap \left(\Oi \cup B^-\right)=\emptyset\right)
 &\le |\ball{0}{2N}|^{1/2}\left(1-\delta_{N,x}^{-c} \rad^{-2}\right)^t\\
 &\le \exp\left\{-\delta_{N,x}^{-c/2} \rad^{-2} t\right\}.
\end{split}
\end{equation}

The proof of~\eqref{eq:forest} is almost identical to that of~\cite[Lemma~1]{Pov99}. We only recall a brief outline of the argument. The key element is~\cite[Proposition~2.4 on pp.160--161]{Szn98}, which roughly says that for a set $U\subset \Z^d$, the eigenvalue $\lambda_{U\setminus\Oi}$ is large if the clearing set $\mathscr{U}_{\rm cl}$ is \emph{locally thin} in the following sense:
\begin{equation}
\left|U\cap \mathscr{U}_{\rm cl} \cap \rad(q+[0,1]^d)\right|=o(\rad^d)\text{ for any }q\in\Z^d.
\end{equation}
On~\eqref{eq:symmdiff1}, we know that $\mathscr{U}_{\rm cl}$ largely coincides with $B^-$ and hence $\ball{0}{2N}\setminus (B^-\cup\Oi)$ is locally thin in the above sense.
\end{proof}

\section*{Acknowledgement}
The authors are grateful to Alain-Sol Sznitman for useful discussions and encouragement on the topic.
J.~Ding is supported by NSF grant DMS-1757479 and an Alfred Sloan fellowship.
R.~Fukushima is supported by JSPS KAKENHI Grant Number 16K05200 and ISHIZUE 2019 of Kyoto University Research Development Program. 
R.~Sun is supported by NUS Tier 1 grant R-146-000-253-114.


\end{document}